\documentclass[11pt]{amsart}
\usepackage[utf8]{inputenc}
\usepackage{amsmath}
\usepackage{amssymb}
\usepackage{amsfonts}
\usepackage{amsthm}
\usepackage[english]{babel}
\usepackage{amsthm}
\usepackage{hyperref}
\usepackage{bm}
\usepackage{array}
\usepackage{float}
\usepackage{tikz-cd}
\usepackage[margin=.94in]{geometry}
\usepackage{etoolbox}

\numberwithin{equation}{section}

\newtheorem{thm}[equation]{Theorem}
\newtheorem*{thm*}{Theorem}

\newtheorem{prop}[equation]{Proposition}

\newtheorem{cor}[equation]{Corollary}

\newtheorem{lemma}[equation]{Lemma}

\theoremstyle{definition}
\newtheorem{ex}[equation]{Example}

\newtheorem{df}[equation]{Definition}

\newtheorem{con}[equation]{Construction}

\newtheorem{notn}[equation]{Notation}

\newtheorem{rmk}[equation]{Remark}

\theoremstyle{definition}
\swapnumbers

\DeclareMathOperator{\set}{\mathbf{Set}}
\DeclareMathOperator{\cat}{\mathbf{Cat}}

\DeclareMathOperator{\ncat}{\mathbf{n-Cat}}
\DeclareMathOperator{\sset}{\mathbf{sSet}}

\DeclareMathOperator{\Hom}{Hom}
\DeclareMathOperator{\ehom}{\underline{Hom}}
\DeclareMathOperator{\fun}{Fun}
\DeclareMathOperator{\sps}{\mathbf{sPSh}}
\DeclareMathOperator{\ps}{\mathbf{PSh}}
\DeclareMathOperator{\map}{map}
\DeclareMathOperator{\hmap}{hmap}
\DeclareMathOperator{\heq}{heq}
\DeclareMathOperator{\Map}{Map}

\DeclareMathOperator{\id}{id}
\DeclareMathOperator{\pr}{pr}
\DeclareMathOperator{\ho}{Ho}
\DeclareMathOperator{\hofiber}{hofiber}
\DeclareMathOperator{\se}{Se}
\DeclareMathOperator{\cpt}{Cplt}
\DeclareMathOperator{\ob}{ob}
\DeclareMathOperator{\diag}{diag}

\DeclareMathOperator*{\hocolim}{hocolim}

\DeclareMathOperator*{\amg}{\amalg}
\DeclareMathOperator*{\tim}{\times}

\newcommand{\op}{^{\text{op}}}

\newcommand{\si}[1]{^{\text{#1}}}

\pretocmd{\subsection}{\stepcounter{equation}}{}{}

\title{Completions and DK-equivalences of $\Theta_n$-spaces}
\author{Miika Tuominen}
\date{}

\begin{document}

\begin{abstract}
We establish Rezk completion functors for $\Theta_n$-spaces with respect to each and all of the completeness conditions. As a consequence, we obtain a characterization completeness of Segal $\Theta_n$-spaces as locality with respect to higher-dimensional Dwyer-Kan equivalences.
\end{abstract}

\maketitle

\tableofcontents

\section{Introduction}

Categories have become ubiquitous throughout mathematics, and in increasing number of contexts, they also come equipped with the additional structure of higher-dimensional morphisms. Especially in higher-dimensional contexts, the composition operations of morphisms are often associative and unital only in a weak sense, which is most conveniently described in the language of homotopy theory.
The $n$-dimensional $(\infty,n)$-categories are prototypically defined as categories enriched in a suitable notion of $(\infty,n-1)$-categories, where $(\infty,0)$-categories should be understood as spaces. However, the strictness at the object level of enriched categories is not well-behaved with respect to homotopy theory, so different models based on presheaves are often used instead.

In the one-dimensional case of $(\infty,1)$-categories, one of more studied models is given by Rezk's complete Segal spaces \cite{rezksp}, which are simplicial presheaves on the simplex category $\Delta$ subject to \emph{Segal} and \emph{completeness} conditions. The Segal condition alone is sufficient to define the structure of a weak composition, but the resulting theory has many different notions of invertibility for morphisms. In order to obtain a homotopy theory for $(\infty,1)$-categories, an additional condition of \emph{completeness} has to be imposed that forces invertibility in terms of homotopical and categorical structures to coincide.

In the category of simplicial spaces, it is common for categorical constructions, such as nerves, in their naïve form to satisfy the Segal condition but fail to be complete; for example, the nerve of a category, although an $(\infty,1)$-category in the quasi-category model on simplicial sets, is just a Segal space that is usually not complete as a discrete simplicial space.
However, the issue of incompleteness for Segal spaces may often be solved via the use of Rezk's \emph{completion} functor \cite[\S 14]{rezksp}, which is a localization construction that is in many cases explicit in form. As an example, the completion applied to the aforementioned discrete nerve gives the homotopy coherent nerve known as the \emph{Rezk nerve} or the \emph{classifying diagram} \cite[3.5]{rezksp}. 

The completion also has the additional property that the natural weak equivalence from a Segal space to its completion is a \emph{DK-equivalence}, which means that it is essentially surjective and fully faithful in a suitably weak sense, preserving objects up to equivalence and the enriched structure in spaces. 
This property creates a link between weak equivalences in the complete Segal space model and DK-equivalences, allowing us to better understand the completeness condition and DK-equivalences between general Segal spaces. DK-equivalences mimic the behaviour of equivalences of categories enriched in spaces and serve as weak equivalences also in other models of $(\infty,1)$-categories based on Segal spaces \cite{bthree,mnsegal}.
In some context, especially ones adjacent to computer science and logic, completeness is also known by the term \emph{univalence}. In this direction Rezk's completion and its generalizations have been extensively studied for their relevance to the program on univalent foundations, building on work in \cite{abmuni}.

In order to model $(\infty,n)$-categories for a general $n$, one may use the category of simplicial presheaves on Joyal's cell category $\Theta_n$ \cite{joyaltheta}, which is an $n$-categorical analogue of the simplex category $\Delta$. Imposing Segal and completeness conditions for each of the $n$ dimensions of morphisms gives the $\Theta_n$-space model for $(\infty,n)$-categories developed by Rezk in \cite{rezkth}. Like in the 1-dimensional case, many constructions in the category of $\Theta_n$-spaces satisfy the Segal conditions but fail some or all of the completeness conditions; for example, framed tangles studied in the context of topological quantum field theories form a Segal $\Theta_n$-spaces that is generally not complete \cite[3.12]{afcob}, \cite[0.26]{afflag}.
In the present paper, we generalize Rezk's completion functor to each of the $n$ completeness conditions of Segal $\Theta_n$-spaces, expressed in terms of a higher-categorical extension of the notion of DK-equivalences.

\begin{thm}\label{t1}
For each $1\leq k\leq n$ there is a functor localizing Segal $\Theta_n$-spaces with respect to the completeness condition in dimension $k$ via a DK-$n$-equivalence.
\end{thm}
In somewhat more precise terms, here by localization we mean a fibrant replacement in a certain model structure on the category of $\Theta_n$-spaces. The notion of DK-$n$-equivalence is an adaptation equivalences of $(\infty,n)$-categories to the incomplete setting and consists of weakened essential surjectivity condition for cells below dimension $n$ together with fully faithfulness on $n$-cells, discussed more extensively in Section \ref{hdk}. We prove Theorem \ref{t1} in its more precise formulation as Theorem \ref{suspended}.

Rezk's completion functor takes a Segal spaces to and equivalent complete one by thickening the space of objects with higher dimensional simplices that encode information about the invertible morphisms while keeping the space of morphisms between any fixed pair of objects unchanged. When $n,k\geq 2$, our completion functor acts similarly $k$ dimensions higher but also leaves morphisms of dimension $k-2$ and lower essentially unchanged. Thus the completion in dimension $k$ essentially only affects morphisms in the single dimension of $k-1$, which allows us to combine the completions for multiple values of $k$ into a single completion functor.

\begin{thm}\label{t2}
There is a fibrant replacement functor localizing Segal $\Theta_n$-spaces with respect to all of the completeness conditions via a DK-$n$-equivalence.
\end{thm}
We define this combined completion functor more precisely in Construction \ref{tcompletion} and discuss its properties in Theorem \ref{total}. 

As in the 1-dimensional setting, DK-$n$-equivalences here encode the prototypical behaviour of equivalences of $(\infty,n)$-categories, and for Segal $\Theta_n$-spaces with completeness above dimension 1, they coincide with prior notions of DK-equivalence, studied for simplicial objects by Berger-Rezk \cite{brcomp2} and Gepner-Haugseng \cite[\S 5]{gh}. For $\Theta_n$-spaces, DK-equivalences have been shown to coincide with equivalences of $(\infty,n)$-categories by Bergner in \cite[6.4]{bdisc} via comparison to the simplicial setting. Our completion functors allow us to provide a direct proof of this correspondence while also relaxing the remaining assumptions on completeness.

\begin{thm}\label{t3}
A map between Segal $\Theta_n$-spaces is a DK-$n$-equivalence if and only if it is becomes a weak equivalence after localizing with respect to all of the completeness conditions; in particular, a Segal $\Theta_n$-space is complete if and only if it is local with respect to the class of DK-equivalences.
\end{thm}
A particularly important implication of this theorem is that the $(\infty,1)$-category of $(\infty,n)$-categories may be obtained from the $(\infty,1)$-category of Segal $\Theta_n$-spaces by localizing with respect to DK-$n$-equivalences. We prove this result in a model-categorical framing as Theorem \ref{dktotal}.


\subsection{Overview}
We use an inductive approach to prove Theorem \ref{t1}, and for the base case $k=1$, our proof strategy follows the same structure as Rezk's for $n=1$. Rezk's completion construction makes use of the Cartesian structure on the category of simplicial spaces; however, DK-equivalences and the weak equivalences of the model structure for complete Segal spaces are not generally compatible with the Cartesian structure when considered between objects that are not fibrant, that is, complete Segal spaces. In order to show that the inclusion of a Segal space into its thickening is both a DK-equivalence and a weak equivalence, Rezk develops a stronger notion of equivalence called \emph{categorical equivalence} that implies the other two and is compatible with the Cartesian structure. Our completion functors for higher values of $k$ is obtained from the base case via a suspension construction that increases the dimension of all morphisms by 1.

In Section \ref{sbg} we discuss the background on the relevant categories and model structures as well as some of the relationships to lower-dimensional cases. In Section \ref{smap} we recount some key properties of the enriched structure  of Segal objects that are essential for studying fully faithfulness part of DK-equivalences, and that allow us to set up the inductive arguments for the higher-dimensional completions. In Section \ref{sho} we discuss notions of homotopy determined by interval objects, which we use to generalize categorical equivalences to a more general setting. This framework also allows us to compare categorical equivalences to weak equivalences and equivalences of categories. In Section \ref{sdk} we study DK-equivalences and their relationship to weak and categorical equivalences. In Section \ref{shcomp} we define the inductive base case of our completion functors and prove Theorem \ref{t1} for $k=1$. We also show a more restricted version of Theorem \ref{t3} which too serves as a base case for induction. The following diagram illustrates the various notions of equivalence that we consider for the inductive base cases of our results with the base case of Theorem \ref{t3} in the centre:
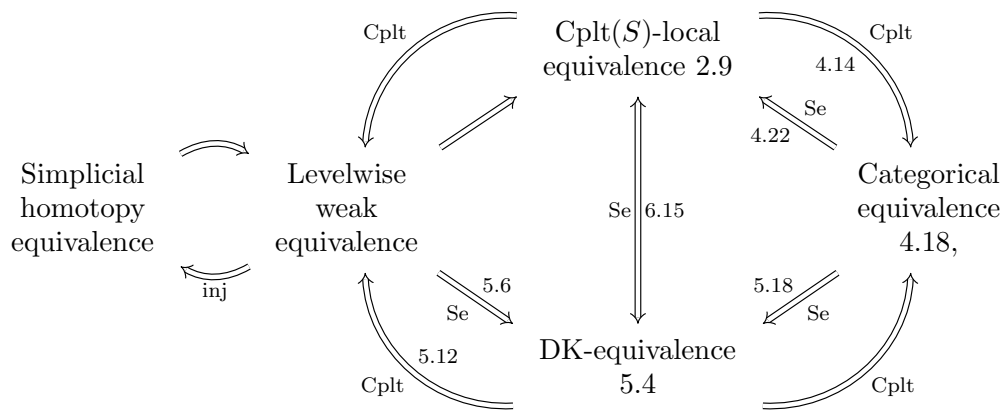
\[
\begin{tikzcd}
& & \begin{tabular}{c}
$\cpt(S)$-local\\
equivalence \ref{cptnotn}
\end{tabular}
\arrow[dd, Leftrightarrow, "\ref{dkmain}", "\se" swap]
\arrow[ld, Rightarrow, bend right=45, "\cpt" swap] 
\arrow[rd, Rightarrow, bend left=45, "\cpt", "\ref{htpywe}" swap] 
 &\\
\begin{tabular}{c}
Simplicial\\
homotopy\\
equivalence
\end{tabular}
\arrow[r, Rightarrow, bend left] & 
\begin{tabular}{c}
Levelwise\\
weak\\
equivalence
\end{tabular}
 \arrow[ru, Rightarrow, end anchor=south west] 
 \arrow[rd, Rightarrow, end anchor=north west, "\se" swap, "\ref{lwdk}"] 
 \arrow[l, Rightarrow, bend left, "\text{inj}"] &  & \begin{tabular}{c}
Categorical \\
equivalence \\ \ref{cateqdef},
\end{tabular} 
\arrow[ld, Rightarrow, end anchor=north east, "\se" , "\ref{cateqdk}" swap]  
\arrow[lu, Rightarrow, end anchor=south east, "\se" swap, "\ref{cateqcpt}" ]\\
& & 
\begin{tabular}{c}
DK-equivalence 
\arrow[lu, Rightarrow, bend left=45, "\cpt", "\ref{dklw}" swap] 
\arrow[ru, Rightarrow, bend right=45, "\cpt" swap] 
\\
\ref{dkdef} 
\end{tabular}
&
\end{tikzcd}
\]
where the implications with labels "inj", "$\se$", and "$\cpt$" hold for injective fibrant, Segal, and Complete Segal objects, respectively. The numbers in the diagram indicate the relevant definition or result showing the implication, whereas the implications without numbers are either well-known or follow from the others.

In Section \ref{hdk} we discuss higher-dimensional extensions for the notions of categorical and DK-equivalence, which we require for characterizing our higher-dimensional completions. Finally, in Section \ref{sscomp} we define our completion functors in the general case and prove our main theorems.

\subsection*{Acknowledgements}
Firstly I want to thank Julie Bergner for introducing me to the problems studied in this paper and guiding me during my doctoral studies. I also want to thank Tim Campion for many insightful discussions related to the topics of this paper and Lyne Moser for very helpful comments on several versions of this text.

\section{Preliminaries}\label{sbg}

In \cite{rezkth}, where $\Theta_n$-spaces are first studied as a model for $(\infty,n)$-categories, the localizations of the injective model structure encoding Segal conditions and completeness are considered inductively, one dimension at a time. In Rezk's approach these localizations are also independent of the previous ones on the higher dimensional structure and thus conveniently described in terms of a more general \emph{$\Theta$-construction} rather than $\Theta_n$ directly. The same level of generality is convenient for our inductive arguments as well, so we adopt a similar approach. The $\Theta$-construction was first introduced by Berger in \cite[3.1.]{bergerth} as a categorical wreath product, but the following description is due to Rezk \cite[3.2]{rezkth}. 

\begin{df}
Let $C$ be a small category. Define a category $\Theta C$ with objects $[m](c_1,\ldots,c_m)$ where $m\in \ob(\Delta)$ and $c_1,\ldots,c_m\in \ob(C)$. A morphism $[m](c_1,\ldots,c_m)\to [l](d_1,\ldots,d_l)$ in $\Theta C$ consists of 
\begin{enumerate}
\item a morphism $\varphi\colon [m]\to[l]$ in $\Delta$, and 
\item morphisms $f_{ij}\colon c_i\to d_j$ for $i=1,\ldots, m$ and $\varphi(i-1)<j \leq \varphi(i)$.
\end{enumerate}
\end{df}

In particular, iterating the $\Theta$-construction on the terminal category, which we denote by $*$, produces Joyal's cell categories $\Theta_n\cong \Theta^n(*)$. We now describe this special case to build intuition for the $\Theta$-construction.

Consider $\Theta_1\cong \Delta$ as a full subcategory of the category of small categories $\cat$, and let $\Sigma$ be the \emph{suspension} functor that takes a strict $(n-1)$-category $A$ to the strict $n$-category with two objects 0 and 1, and 
$\Hom$-$(n-1)$-categories
\[
\begin{split}
\ehom_{\Sigma A}(0,1)&= A,\\
\ehom_{\Sigma A}(0,0)&=\{\id_0\},\\
\ehom_{\Sigma A}(1,1)&=\{\id_1\},\quad \text{and}\\
\ehom_{\Sigma A}(1,0)&=\emptyset.
\end{split}
\]
In essence, $\Sigma$ introduces two new 0-cells and raises the dimension of all cells of $A$ by 1.
We may then consider $\Theta_n$ as the full subcategory of the category strict $n$-categories $\ncat$ on objects
\[
[m](\theta_1,\ldots,\theta_m)\cong\Sigma \theta_1 \amg_{[0]} \Sigma \theta_2 \amg_{[0]} \cdots \amg_{[0]} \Sigma \theta_m,
\]
where each of the pushouts is with respect to the inclusion at 1 in the preceding term and the inclusion at 0 in the following term, and $m\in \ob (\Delta)$ and $\theta_i\in \ob (\Theta_{n-1})$.
As an example the object $[3]([2],[0],[1])\in \Theta_2$ is the 2-category generated by the data
\[
  \begin{tikzcd}
    0 \arrow[bend left=60]{r}[name=LUU, below]{}
    \arrow{r}[name=LUD]{}
    \arrow[swap]{r}[name=LDU]{}
    \arrow[bend right=60]{r}[name=LDD]{}
    \arrow[Rightarrow,to path=(LUU) -- (LUD)\tikztonodes, shorten=-2pt]{r}
    \arrow[Rightarrow,to path=(LDU) -- (LDD)\tikztonodes, shorten=-2pt]{r}
    & 
    1
    \arrow{r}[name=RUD]{}
    & 
    2
    \arrow[bend left=40]{r}[name=RrUU, below]{}
     \arrow[bend right=40]{r}[name=RrUD]{}
    \arrow[Rightarrow,to path=(RrUU) -- (RrUD)\tikztonodes, shorten=1pt]{r}
    & 
    3.
  \end{tikzcd}
\]
It may be useful to think of the objects of the more general $\Theta C$ as chains of suspended objects from $C$.
We refer the reader to \cite[\S 3]{rezkth} for further discussion on the $\Theta$-construction and to \cite[\S 4]{orqa} for discussion on its connection to the suspension.

\begin{notn}
Throughout this paper we let $C$ be a small category with a fixed terminal object $t$.

Let $\ps(C)=\fun(C\op,\set)$ denote the set valued presheaves and $\sps(C)=\fun(C\op,\sset)$ the simplicial presheaves on $C$. 
We write $F\colon C\to \sps(C)$ for the \emph{discrete Yoneda embedding}, that is, the usual Yoneda embedding $C\to \ps(C)$ followed by postcomposition by the discrete inclusion $\set\to \sset$. We may also view elements of $\sps(C)$ interchangeably as elements of $\ps(C\times \Delta)$.

In the case $C=\Theta_n$, the inclusion $\Theta_n\to \ncat$ has an extension along the discrete Yoneda embedding which we call the \emph{discrete nerve} $N^n\colon \ncat\to \sps(\Theta_n)$, given by the levelwise formula 
\[
N^n(A)_{\theta,p}=\Hom_{\ncat}(\theta, A),
\]
for $\theta\in \Theta_n\subset \ncat$, $[p]\in \Delta$ and $A\in \ncat$.
When $n=1$ we omit the superindex and write $N=N^1$.

We use arrows like $\hookrightarrow$ for cofibrations, $\twoheadrightarrow$ for fibrations, and decorate morphisms that are weak equivalences with $\sim$ when relevant. Which model structure these types of morphisms are considered with respect to varies by context.
\end{notn}

Recall that the category $\sps(C)$ has the \emph{injective} model structure, where weak equivalences are levelwise weak equivalences of simplicial sets in the model structure for Kan complexes and cofibrations are monomorphisms. 
We also note that by a result of Bergner and Rezk \cite{brreedy} the injective model structure on $\sps(\Theta_n)$ coincides with the \emph{Reedy} model structure, where fibration may be described more explicitly, although we do not need that description for our results.

The $\Theta$-construction may be extended from the strict setting to the weak one via the discrete Yoneda embedding as follows.

\begin{df}\cite[4.4]{rezkth}
The \emph{intertwining functor} $V\colon \Theta (\sps(C))\to \sps(\Theta C)$ is the left Kan extension of the discrete Yoneda embedding of $\Theta C$ along the $\Theta$-construction of the discrete Yoneda embedding of $C$ as illustrated in the diagram
\[
\begin{tikzcd}
\Theta C \arrow[r,"F_{\Theta C}"] \arrow[d,"\Theta(F_{C})" left] & \sps(\Theta C).\\
\Theta(\sps(C)) \arrow[ru,"V" below right]& \\
\end{tikzcd}
\]
\end{df}
The intertwining functor has the levelwise formula
\begin{equation}\label{veq}
V[n](X_1,\ldots,X_n)_{[m](c_1,\ldots,c_m)}=\coprod_{\delta\in \Hom_{\Delta}([m],[n])}\prod_{i=1}^m \prod_{j=\delta(i-1)+1}^{\delta(i)} (X_j)_{c_i},
\end{equation}
from which we can observe that the intertwining functor preserves monomorphisms in each variable, since the finite products and coproduct in the levelwise formula preserve injectivity.

Of particular interest is the functor $V[1]\colon \sps(C)\to (F[0]\coprod F[0])/\sps(\Theta C)$, which extends the suspension $\Sigma$ when $C=\Theta_n$, but is well-defined for any $C$.

\begin{prop}\cite[4.6]{rezkth}\label{vleftq}
The functor $V[1]$ is left Quillen with respect to the injective model structures.
\end{prop}

In order to describe the Segal and completeness conditions for presheaves on $\Theta C$, we need functors relating them to simplicial spaces where the conditions are well understood. However, higher-dimensional comparisons are also of use to us for describing completeness in higher dimensions.

Fix the category $C$, let $n\geq k\geq 0$, and consider the unique functor $\pi_0^{n-k}\colon \Theta^{n-k} C\to *$ to the terminal category; since $\Theta^{n} C$ has a terminal object, the functor $\pi_0^{n}$ has a right adjoint $\tau_0^{n-k}\colon * \to \Theta^{n-k} C$ given by the constant functor at the terminal object.
Applying the $\Theta$-construction $k$ times to these functors then gives an adjunction
\[
  \begin{tikzcd}
  \Theta^{n} C
    \arrow[r, bend left=40, "\pi_k^{n}"{name=mu, above}] 
    & 
    \Theta_k,
    \arrow[l, bend left=40, "\tau_k^{n}"{name=md, below}]  
    \arrow[l, phantom,to path=(mu) -- (md)\tikztonodes, "\perp" ]
    \end{tikzcd}
\]
where $\pi_k^{n}$ forgets cells above dimension $k$ and $\tau_k^{n}$ adds only identities above dimension $k$.

Then recall that given a functor $f\colon C \to D$ between small categories, the precomposition functor $f^*\colon \sps(D)\to \sps(C)$ has a left adjoint $f_!$ and a right adjoint $f_*$ given by left and right Kan extension, respectively. Additionally, given an adjoint pair of functors, the associated precomposition functors also form an adjoint pair. In particular, we have a chain of adjoint functors

\[
  \begin{tikzcd}
  \sps(\Theta_k)
    \arrow[r, bend left=20, "(\pi_k^n)^*\cong (\tau_k^n)_!"{name=mu, above}] 
    \arrow[r, bend right=70, looseness=1.5, "(\tau_k^n)_*" below, ""{name=d, above}]   
    & 
    \sps(\Theta^n C),
    \arrow[l, bend left=20, "(\pi_k^n)_*\cong (\tau_k^n)^*"{name=md, below}]
    \arrow[l, bend right=70, looseness=1.5, "(\pi_k^n)_!" above, ""{name=u, below}]
    \arrow[phantom,to path=(u) -- (mu)\tikztonodes, "\perp" near start]
    \arrow[phantom,to path=(md) -- (d)\tikztonodes, "\perp" near end]
    \arrow[phantom,to path=(mu) -- (md)\tikztonodes, "\perp" ]
    \end{tikzcd}
\]
where the natural isomorphisms follow by uniqueness of adjoints.
The functor $(\tau_k^n)^*$ may be viewed as forgetting cells above dimension $k$ and $(\pi_k^n)^*$ as the inclusion giving only trivial structure above dimension $k$. When $n=k=1$ we omit the indices and write $\tau =\tau_1^1$ and $\pi = \pi_1^1$.

\begin{df}
Given a presheaf $X \in \sps(\Theta^n C)$, we call $(\tau_k^{n})^* X\in \sps(\Theta_k)$ the \emph{underlying $\Theta_k$-space} of $X$, or the \emph{underlying simplicial space} of $X$ in the case $k=1$.
\end{df}

Note that the functors $(\pi_k^{n})^*$ may be used to extend the discrete nerves; in particular, we obtain a nerve functor of 1-categories
\[
\pi^*N\colon \cat \to\sps(\Delta)\to \sps(\Theta C).
\]

Next we describe the Segal and completeness condition and the associated model structures in the language of Bousfield localizations. 

Recall that the category of simplicial presheaves $\sps(C)$ has a simplicial structure with mapping spaces given by
\[
\Map(X,Y)_p= \Hom(X \times \Delta[p], Y),
\]
where $\Delta[p]$ is the representable simplicial set on $[p]\in \Delta$ viewed as a constant presheaf.

\begin{df}
Let $S$ be a set of morphisms in $\sps(C)$. We say that an object $Z$ in $\sps(C)$ is \emph{$S$-local} if
\[
f^*\colon \Map(Y,Z)\to \Map(X,Z)
\]
is a weak equivalence for all $f\colon X\to Y$ in $S$, and we say that $Z$ is \emph{$S$-fibrant} if it is $S$-local and injective fibrant.

Furthermore, we call a morphism $g\colon X\to Y$ an \emph{$S$-local equivalence} if the morphism
\[
g^*\colon \Map(Y,Z)\to \Map(X,Z)
\]
is a weak equivalence for every $S$-fibrant $Z$.
\end{df}

The following theorem states the existence of a model structure, which may be viewed as being obtained from the injective one by forcing elements of $S$ to be weak equivalences in addition to levelwise weak equivalences.

\begin{thm}\cite[4.1.1.]{hirsch}
Let $S$ be a set of morphisms in $\sps(C)$. Then there is a model structure on $\sps(C)$ where the cofibrations are monomorphisms, the weak equivalences are $S$-local equivalences and the fibrant objects are $S$-fibrant objects.
\end{thm}
We refer to this model structure as the \emph{$S$-local model structure}.

The first collection of maps in $\sps(\Theta C)$ that we want to localize with respect to are the \emph{Segal maps}
\[\begin{split}
\se_C=\{ \operatorname{se}^{[m](c_1,\ldots , c_m)} \colon F([1](c_1))\amg_{F[0]}\cdots \amg_{F[0]} F([1](c_m))\hookrightarrow F([m](c_1,\ldots , c_m))  | \\ [m]\in \Delta, c_1,\ldots , c_m \in C \},
\end{split}\]
which are induced by the \emph{spine inclusions}
\[
[1]\amg_{[0]}\cdots \amg_{[0]} [1] \hookrightarrow [m],
\]
where on the left-hand-side the endpoint of each arrow $[1]$ is glued to the start of the next, forming a chain of $m$ composable morphisms. For illustration, consider the case $C=*$, where each Segal map is directly obtained by applying the discrete nerve $N\colon \cat \to \sps(\Delta)$ to the colimit diagram defining the corresponding spine. Note that the spine inclusions are isomorphisms in $\cat$, since there is a unique composite for any chain of $m$ morphisms. However, the discrete nerve does not preserve colimits, so the Segal maps are not isomorphisms or even levelwise weak equivalences, but by localizing with respect to them, we recover the structure of a composition operation in a weaker, up to homotopy form. Allowing for general $C$ adds enrichment over $\sps(C)$. 

Although $\se_C$-fibrant objects have a weak category-like structure, their homotopy theory fails to capture equivalences within and thus also between such categorical structures. However, by forcing an appropriate model for a walking equivalence to be contractible, we may impose a relationship between objects and equivalences within the categorical structure of $\se_C$-fibrant objects that resolves the issue. In the strict setting of $\cat$ equivalences may be encoded using the walking isomorphism $I$, and the precomposition by the functor $I\to *$ induces the inclusion of identities to isomorphisms. Using the discrete nerve $\pi^*N$, we may pass to $\sps(\Theta C)$ where the composition operation is weaker, making $\pi^*E := \pi^*N I$ the walking 1-equivalence. 

The \emph{completeness} map is defined as the discrete nerve applied to the functor $I\to *$:
\[
\operatorname{cplt}_C= \{\pi^*E \to F[0]\}.
\]
For the Segal maps together with the completeness map, we write
\[
\cpt_C=\operatorname{cplt}_C\cup \se_C.
\]
The maps in $\cpt_C$ only encode information regarding 1-dimensional structure in $\sps(\Theta C)$ introduced by the $\Theta$-construction, and in order to additionally impose structure encoded in $C$, we may consider a set of morphisms $S$ in $\sps(C)$ and suspend it to raise it in dimension.
\begin{notn}\label{cptnotn}
We denote $\se(S)=\se_C \cup V[1](S) $ and $\cpt(S)=\cpt_C \cup V[1](S)$.

In particular in the presence of the iterated $\Theta$-construction, we may suspend the Segal and completeness maps to encode higher dimensional weak categorical structure. For Segal and combined Segal and completeness maps in multiple consecutive dimensions we denote $\se^n(S)= \se(\se^{n-1}(S))$ and $\cpt^n(S)= \cpt(\cpt^{n-1}(S))$ for brevity.
More explicitly, we have 
\[
\se^n(S)=\left(\bigcup_{k=1}^n \se_{\Theta^{n-k} C}\right) \cup V[1]^n(S)
\]
and 
\[
\cpt^n(S)=\left(\bigcup_{k=1}^n V[1]^{k-1} (\operatorname{cplt}_{\Theta^{n-k} C})\right) \cup \left(\bigcup_{k=1}^n V[1]^{k-1} (\se_{\Theta^{n-k} C})\right)   \cup V[1]^n(S).
\]
We refer to the locality with respect to $V[1]^{k-1}(\operatorname{cplt}_{\Theta^{n-k} C})$ as \emph{completeness in dimension $k$}.
\end{notn}

\begin{df}
In the category $\sps(\Theta^n C)$, we call $\se^n(S)$-fibrant objects \emph{Segal} objects and $\cpt^n(S)$-fibrant objects \emph{complete Segal} objects
\end{df}

We recall a key property of the model structures that we consider in this paper.
\begin{df}
A category $\mathcal{C}$ is \emph{Cartesian closed} if it has finite products and for each object $X$ of $\mathcal{C}$ the functor $- \times X$ has a right adjoint $-^X$ contravariant in $X$.

A model category $\mathcal{M}$ is \emph{Cartesian} if it is Cartesian closed as a category and the two equivalent conditions hold.
\begin{enumerate}
\item If $f\colon X\hookrightarrow Y$ and $g\colon Z\hookrightarrow W$ are cofibrations in $\mathcal{M}$, then the map induced by $f$ and $g$
\[
X\times W \amg_{X\times Z} Y\times Z \to Y\times W
\]
is a cofibration.
\item If $f\colon X\twoheadrightarrow Y$ is a fibration and
 $g\colon Z\hookrightarrow W$ is a cofibration in $\mathcal{M}$, then the map induced by $f$ and $g$
\[
X^W\to 
X^Z \tim_{Y^Z} Y^W 
\]
is a fibration.
\end{enumerate}
\end{df}

\begin{prop}\cite[8.1, 8.5]{rezkth}
Let $S$ be a set of morphisms in $\sps(C)$. If the $S$-local model structure on $\sps(C)$ is Cartesian, then so are the $\se(S)$-local and $\cpt(S)$-local model structures on $\sps(\Theta C)$.
\end{prop}

\begin{notn}
In the remainder of this paper we let $S$ be a set of morphisms in $\sps(C)$ such that $S$-local model structure is Cartesian.
\end{notn}

The following proposition is a key tool for studying Quillen equivalences between localized model structures.

\begin{prop}\cite[2.16]{rezkth}\label{qad}
Let
\[
  \begin{tikzcd}
  \sps(\Theta C)
    \arrow[r, bend left=20, "L"{name=mu, above}] 
    & 
    \sps(\Theta C')
    \arrow[l, bend left=20, "R"{name=md, below}]  
    \arrow[l, phantom,to path=(mu) -- (md)\tikztonodes, "\perp" ]
    \end{tikzcd}
\]
be a Quillen pair with respect to the injective model structures on $\sps(\Theta C)$ and $\sps(\Theta C')$, and let $S$ and $S'$ be sets of morphisms in $\sps(\Theta C)$ and $\sps(\Theta C')$, respectively. Then $L \dashv R$ is a Quillen pair with respect to the $S$-local and $S'$-local model structures if and only if the two equivalent conditions are satisfied.
\begin{enumerate}
\item The map $L(s)$ is a $S'$-local equivalence for each $s\in S$.
\item The object $R(X)$ is $S$-fibrant for all $S'$-fibrant $X$ in $\sps(\Theta C')$.
\end{enumerate}
\end{prop}

\begin{cor}
The functor $V[1]$ is left Quillen from the $S$-local model structure to the $\se(S)$-local model structure.
\end{cor}

\begin{rmk}
The adjunction $(\pi_k^n)^* \dashv (\tau_k^n)^*$ forms a Quillen pair with respect to the injective, $\se^k$-local, and $\cpt^k$-local model structures.
\end{rmk}

Localizing with respect to the Segal maps also forces other similar maps to be weak equivalences.

\begin{prop}\cite[5.3]{rezkth}\label{moresegal}
Let $X_1,\ldots,X_m$ be objects of $\sps(C)$ then the maps
\[
V[1](X_1)\amg_{V[0]}\cdots \amg_{V[0]} V[1](X_m)\mapsto V[m](X_1,\ldots,X_m)
\]
are $\se$-local equivalences for all $m$.
\end{prop}

There is also an alternate description of the Segal condition: using the Yoneda lemma we obtain isomorphisms 
$\Map(F([1](c_i),X)\cong X_{[1](c_i)}$ and $\Map(F([m](c_1,\ldots , c_m),X)\cong X_{[m](c_1,\ldots , c_m)}$, and thus $X\in \sps(\Theta)$ is $\se$-local if and only if the maps
\[
X_{[m](c_1,\ldots , c_m)} \to X_{[1](c_1)}\tim_{X_{[0]}}\cdots \tim_{X_{[0]}} X_{[1](c_m)}
\]
induced by the spine inclusions are weak equivalences.

Using the following classical result regarding homotopy pullbacks, we may also note that the pullbacks on the above are, in fact, homotopy pullbacks if $X$ is injective fibrant, since the pullbacks are with respect to maps induced by monomorphisms $F[0]\hookrightarrow F[1](c_i)$.

\begin{prop}\label{fibpb}\cite[A.2.4.4.]{luriehtt}
Let the following be a pullback square in a model category $\mathcal{M}$:
\[
\begin{tikzcd}
X \arrow[r] \arrow[d] \arrow[rd, phantom, "\lrcorner" very near start]
& Z \arrow[d, "f"]\\
Y \arrow[r] & W. \\
\end{tikzcd}
\]
If $f$ is a fibration, and $Y$ and $W$ are fibrant, then the diagram is also a homotopy pullback square.
\end{prop}

The Segal condition allows us to reduce properties regarding all levels of an object to just the levels that are zero or one 1-cells in length.

\begin{prop}\label{segallw}
Let $f\colon X\to Y$ be a map of $\se(S)$-fibrant objects of $\sps(\Theta C)$. Then $f$ is a levelwise weak equivalence if and only if it is a weak equivalence at level $0$ and levels $[1](c)$ for all $c\in C$.
\end{prop}

\begin{proof}
It suffices to show the backwards direction. Let $f\colon X\to Y$ be a map of $\se(S)$-fibrant objects and suppose that $f_{[0]}$ and $f_{[1](c)}$ are weak equivalences.
Then the map on the iterated homotopy pullbacks 
\[
\begin{tikzcd}
X_{[1](c_1)}\times_{X_{[0]}}^h\cdots \times_{X_{[0]}}^h X_{[1](c_m)} \arrow[r] & Y_{[1](c_1)}\times_{Y_{[0]}}^h\cdots \times_{Y_{[0]}}^h Y_{[1](c_m)}
\end{tikzcd}
\]
is also a weak equivalence for all $c_1,\ldots,c_m$. By injective fibrancy of $X$ and $Y$ these homotopy pullbacks are weakly equivalent to strict pullbacks.
Now we have the diagram
\[
\begin{tikzcd}
X_{[m](c_1,\ldots,c_m)} \arrow[r,"f_{[m](c_1,\ldots,c_m)}" above] \arrow[d, "\sim" left] & Y_{[m](c_1,\ldots,c_m)} \arrow[d, "\sim" ]\\
\displaystyle X_{[1](c_1)}\tim_{X_{[0]}}\cdots \tim_{X_{[0]}} X_{[1](c_m)} \arrow[r,"\sim"] & \displaystyle Y_{[1](c_1)}\tim_{Y_{[0]}}\cdots \tim_{Y_{[0]}} Y_{[1](c_m)}
\end{tikzcd}
\]
with the Segal maps vertically, so $f_{[m](c_1,\ldots,c_m)}$ is also a weak equivalence by the two-out-of-three property. Thus $f$ is a levelwise equivalence.
\end{proof}

Higher-dimensional Segal conditions allow for similar reductions so that, for example, a Segal $\Theta_n$-space $X$ is determined by the levels $[1]^k [0]$, which may be understood us the spaces of $k$-cells in $X$, for $0\leq k \leq n$.

\section{Mapping objects}\label{smap}

We have seen that a Segal object $X$ in $\sps(\Theta C)$ is essentially determined by its levels $[0]$ and $[1](c)$ for $c\in C$, where levels of the latter form may be viewed as encoding information about morphisms in $X$. As with usual categories, morphisms in $X$ can be studied more tractably by fixing a source and a target, but also have additional structure making $X$ enriched in $\sps(C)$. In this section we recall some key properties of these \emph{mapping objects}.

\begin{df}
Let $X$ be an object of $\sps(\Theta C)$, and $x,y\in X_{[0],0}$. We define the \emph{mapping object} of $X$ from $x$ to $y$ as the object of $\sps(C)$ given at level $c$ as the fiber
\begin{equation}\label{mapd}
\begin{tikzcd}
\map_X(x,y)_c \arrow[r] \arrow[d] \arrow[dr, phantom, "\lrcorner" very near start]
& X_{[1](c)} \arrow[d,  "{(d_1,d_0)}"]\\
\{*\} \arrow[r, "{(x,y)}"] & X_{[0]}\times X_{[0]}.
\end{tikzcd}
\end{equation}
We similarly define the \emph{homotopy mapping object} of $X$ from $x$ to $y$ levelwise as the homotopy fiber
\[
\begin{tikzcd}
\hmap_X(x,y)_c \arrow[r] \arrow[d] \arrow[dr, phantom, "\lrcorner_h" very near start]
& X_{[1](c)} \arrow[d,  "{(d_1,d_0)}"]\\
\{*\} \arrow[r, "{(x,y)}"] & X_{[0]}\times X_{[0]}.
\end{tikzcd}
\]
\end{df}
While the homotopy mapping objects are invariant under levelwise weak equivalences, homotopy limits are generally less tangible than strict limits. However, when $X$ is injective fibrant, then the right vertical map in the diagram (\ref{mapd}) is a fibration, which implies that the canonical map $\map_X(x,y) \to \hmap_X(x,y)$ is a levelwise weak equivalence by Proposition \ref{fibpb}.

In the cases where $X$ is not injective fibrant we use the following standard construction for homotopy pullback, which can be found, for example, at \cite[\href{https://kerodon.net/tag/010B}{010B}]{kerodon}.
\begin{lemma}\label{hmapd}
Let $X$ be an object of $\sps(\Theta C)$ such that $X_0\times X_0$ is a Kan complex. Then $hmap_X(x,y)_{c}$ is weakly equivalent to the limit of the diagram
\[
\begin{tikzcd}
& &  X_{[1](c)} \arrow[d]\\
& (X_{[0]}\times X_{[0]})^{\Delta[1]} \arrow[r, "j_0^*" above] \arrow[d, "j_1^*" left] & X_{[0]}\times X_{[0]}\\
\{*\} \arrow[r, "{(x,y)}" ] & X_{[0]}\times X_{[0]}, &
\end{tikzcd}
\]
where $j_0^*,j_1^*$ are the precompositions by the two inclusions $j_0, j_1\colon \Delta[0]\to \Delta[1]$.
\end{lemma}

We may also define a multiobject variant of the mapping objects $\map_X(x_1,\ldots, x_m)\in \sps(C^{\times m})$ levelwise as the fibers
\[
\begin{tikzcd}
\map_X(x_1,\ldots, x_m)_{c_1,\ldots,c_m} \arrow[r] \arrow[d] \arrow[dr, phantom, "\lrcorner" very near start]
& X_{[m](c_1,\ldots,c_m)} \arrow[d]\\
\{*\} \arrow[r, "{(x_1,\ldots, x_m)}"] & X_{[0]}^{m+1}.
\end{tikzcd}
\]
The Segal map 
$
\begin{tikzcd}
X_{[2](c,c)}\to X_{[1](c)}\tim_{X_{[0]}} X_{[1](c)}
\end{tikzcd}
$
can be restricted to the fibers to obtain maps $\map_X(x,y,z)_{c,c}\to \map_X(y,z)_{c}\times \map_X(x,y)_{c}$, which have homotopy inverses natural in $c$ if $X$ is $\se$-fibrant. The Segal object $X$ then obtains a weakly unital and associative composition operation as the composite
\[
\begin{tikzcd}
 \map_X(y,z)\times \map_X(x,y) \arrow[r] & \map_X(x,y,z) \arrow[r] & \map_X(x,z).
\end{tikzcd}
\]
See for example \cite[7.4]{rezkth} for further discussion on the composition.

For an object $X$ in $\sps(\Theta^n C)$ we may iterate the mapping object construction up to $n$ times.
Given $\alpha,\alpha'\in \map_X^{k-1}(\beta, \beta' )_{[0],0}$ and $k=1,\ldots,n$, we denote 
\[
\map_X^k(\alpha,\alpha') := \map_{\map_X^{k-1}(\beta,\beta')}(\alpha,\alpha'),
\]
where $\map_X^{0}=X$. We refer to the objects $\map_X^k(\alpha,\alpha')$ as the $k^{\text{th}}$ \emph{iterated mapping objects} of $X$.

The following lemma establishes an adjunction between the mapping objects and the suspension.

\begin{lemma}\cite[4.12]{rezkth}\label{mapgen}
Let $X$ be an object of $\sps(\Theta C)$ and $A$ an object of $\sps(C)$. Then there is a pullback square
\[
\begin{tikzcd}
\Map(A,\map_X(x,y)) \arrow[r] \arrow[d]
& \Map(V[1](A),X) \arrow[d]\\
\{*\} \arrow[r, "{(x,y)}"] & \Map(V[1](\emptyset),X).
\end{tikzcd}
\]
\end{lemma}

\begin{cor}
The functor 
\[
\map\colon  (F[0]\amg F[0])/\sps(\Theta C) \to \sps(C), 
\]
\[
\left({\begin{tikzcd}
F[0]\coprod F[0] \arrow[r, "{(x,y)}"] & X
\end{tikzcd}}\right) \mapsto  \map_X(x,y)
\]
is right adjoint to $V[1]$ and thus right Quillen with respect to the injective model structures and $\se(S)$-local and $S$-local model structures.
\end{cor}

\begin{cor}
Let $X$ be an injective fibrant object of $\sps(\Theta C)$. Then $\map_X(x,y)$ is injective fibrant in $\sps(\Theta C)$ for all $x,y\in X_{[0],0}$.
\end{cor}

Similarly, $\se(S)$-fibrant objects have $S$-fibrant mapping objects. However, having $S$-fibrant mapping objects is also sufficient for $V[1](S)$-locality for Segal objects.  

\begin{prop}\label{mapfibrant}
Let $X$ be a $\se$-fibrant object of $\sps(\Theta C)$. Then $X$ is $\se(S)$-fibrant if and only if the mapping objects $\map_X(x,y)$ are $S$-fibrant in $\sps(C)$ for all $x,y\in X_{[0],0}$.
\end{prop}

For the proof of this proposition we recall the fiberwise characterization of homotopy pullbacks, which follows from the naturality of the long exact sequence of homotopy groups. See for example \cite[3.3.18]{cubical} for further discussion.
\begin{prop}\label{fiber}
Let the following be a commutative square of simplicial sets:
\[
\begin{tikzcd}
X \arrow[r] \arrow[d, "f" left] & Z \arrow[d, "g"]\\
Y \arrow[r, "h"] & W. \\
\end{tikzcd}
\]
Then the square is a homotopy pullback if and only if the induced map $\hofiber(f)_y\to \hofiber(g)_{h(y)}$ is a weak equivalence for all $y\in Y_{[0]}$.
\end{prop}

\begin{proof}[Proof of Proposition \ref{mapfibrant}]
Let $X$ be $\se$-fibrant and $f\colon Y\to Z$ be in $S$, and consider the diagram
\begin{equation}\label{suspfibrancy}
\begin{tikzcd}
\Map(V[1](Z),X) \arrow[r, "{(V[1](f))^*}"] \arrow[d, two heads]
& \Map(V[1](Y),X) \arrow[d, two heads]\\
\Map(V[1](\emptyset),X) \arrow[r, equal]
& \Map(V[1](\emptyset),X),
\end{tikzcd}
\end{equation}
which is a homotopy pullback if and only if $(V[1](f))^*$ is a weak equivalence, which happens if and only if the maps induced on all homotopy fibers of the vertical arrows are weak equivalences by Proposition \ref{fiber}.
However, due to injective fibrancy of $X$ the homotopy fibers coincide with the fibers, which by Lemma \ref{mapgen} have the form
\[
\begin{tikzcd}
\Map(Z,\map_X(x,y)) \arrow[r, "f^*"]
& \Map(Y,\map_X(x,y)).
\end{tikzcd}
\]
Thus $X$ is $V[1](S)$-local if and only if $\map_X(x,y)$ is $S$-local for all $x$ and $y$.
\end{proof}

The following lemma tells us that the underlying $\Theta_k$-space functor is compatible with the structure of mapping objects.

\begin{prop}\label{umap}
Let $X$ be a an object of $\sps(\Theta^n C)$ and $1\leq m\leq k\leq n$. Then 
\[
\map_{(\tau_k^n)^* X}^m(\alpha, \alpha') \cong (\tau_{k-m}^{n-m})^*\map_X(\alpha, \alpha')
\]
for all $\alpha$ and $\alpha'$.
\end{prop}

\begin{proof}
It suffices to show that the left adjoints commute, that is, 
\[
 V[1]^m((\pi_{k-m}^{n-m})^*Y)\cong (\pi_k^n)^* V[1]^m(Y),
\]
which can be further reduced to the case $m=1$ by iteration.
Using the formula (\ref{veq}) for the suspension we obtain the following natural isomorphisms:
\[
\begin{split}
V[1]((\pi_{k-1}^{n-1})^*Y)_{[m](c_1,\ldots,c_m)} 
& \cong \{*\} \amalg \left( \coprod_{i=1}^m ((\pi_{k-1}^{n-1})^*Y)_{c_i}
\right)\amalg\{*\}\\
& = \{*\} \amalg \left( \coprod_{i=1}^m Y _{\pi_{k-1}^{n-1} c_i}
\right)\amalg\{*\}\\
& \cong V[1](Y)_{[m](\pi_{k-1}^{n-1}c_1,\ldots,\pi_{k-1}^{n-1}c_m)} \\
& = V[1](Y)_{\pi_k^n([m](c_1,\ldots,c_m))} \\
& = (\pi_k^n)^*V[1](Y)_{[m](c_1,\ldots,c_m)}
\end{split}
\]
proving the claim.
\end{proof}

\section{Homotopy relations}\label{sho}

In this section we examine notions of homotopy equivalence on simplicial presheaf categories and how they are related  to equivalences of categories via the discrete nerve functor. The notions of equivalence we are interested in arise as homotopy equivalences in Cartesian model structures, so interval objects provide a good common framework. Most of the results in this section are adapted from \cite[\S 4]{ds} to the Cartesian setting. Interval objects in the presheaf setting have also been extensively studied in \cite[\S 1.3]{cis}, for example.

\begin{notn}
In this section let $\mathcal{M}$ be a Cartesian model category where all monomorphisms are cofibrations. In particular, all objects of $\mathcal{M}$ are cofibrant.
\end{notn}

The three choices for $\mathcal{M}$ that we are interested in are the injective and $\cpt(S)$-local model structures on $\sps(\Theta C)$ and the canonical model structure on $\cat$, which was recorded by Rezk in \cite{rezkcat}. Note that we do not assume that every cofibration in $\mathcal{M}$ is a monomorphism as this is not the case in $\cat$.

\begin{df}\label{intdef}
We call an object $J$ in $\mathcal{M}$ an \emph{interval object} for $\mathcal{M}$ if there is a factorization
\[
\begin{tikzcd}
* \arrow[d, hook] \arrow[drr, equal, bend left=20] \\
*\coprod * \arrow[r, hook] & J \arrow[r, "\sim"] & *.\\
* \arrow[u, hook] \arrow[urr, equal, bend right=20]
\end{tikzcd}
\]
\end{df}
For the remainder of this section let $J$ denote an arbitrary interval object for $\mathcal{M}$.

The following two examples are the standard interval objects for the respective model structures.

\begin{ex}
The constant presheaf $\Delta[1]$ is an interval object for the injective model structure on $\sps(C)$.
\end{ex}

\begin{ex}
The walking isomorphism $I$, which consists of two objects and an isomorphism between them, is an interval object for the canonical model structure on $\cat$.
\end{ex}

Using the Cartesian structure on $\mathcal{M}$, an interval object determines functorial cylinder and path objects in the sense of \cite[\S 4]{ds}.

\begin{lemma}
Let $U$ and $V$ be objects of $\mathcal{M}$ with $V$ fibrant. Then $U\times J$ is a \emph{good cylinder object} for $U$ and $V^J$ is a \emph{very good path object} for $V$, that is, there are factorizations
\[
\begin{tabular}{c c c}
\begin{tikzcd}
U \arrow[d] \arrow[drr, equal, bend left=20] \\
U\coprod U \arrow[r, hook] & U \times J \arrow[r, "\sim"] & U\\
U \arrow[u] \arrow[urr, equal, bend right=20]
\end{tikzcd} & \text{and} & 
\begin{tikzcd}
& & V  \\
V \arrow[r, hook,  "\sim"] \arrow[urr, equal, bend left=20] \arrow[drr, equal, bend right=20]  & 
V^J \arrow[r, two heads] 
& V \times V \arrow[u, "\pr_1" right] \arrow[d, "\pr_2" right]\\
& & V.
\end{tikzcd}
\end{tabular}
\]
\end{lemma}

\begin{proof}
Consider the factorization associated to the interval object $J$
\begin{equation}\label{eqinterval}
\begin{tikzcd}
* \arrow[d, hook] \arrow[dr, hook, "j_1"] \arrow[drr, equal, bend left=20] \\
*\coprod * \arrow[r, hook] & J \arrow[r, "\sim"] & *,\\
* \arrow[u, hook] \arrow[ur, hook, "j_2" swap] \arrow[urr, equal, bend right=20]
\end{tikzcd}
\end{equation}
and note that $j_1$ and $j_2$ are weak equivalences by the two-out-of-three property.

Using the Cartesian property of the model structure on $\mathcal{M}$, we can apply the functors $U\times -$ and $V^{-}$ to the above diagram (\ref{eqinterval}) to obtain the diagrams
\[
\begin{tabular}{c c c}
\begin{tikzcd}
U \arrow[d, hook] \arrow[dr, hook, "\sim"] \arrow[drr, equal, bend left=20] \\
U\coprod U \arrow[r, hook] & U \times J \arrow[r] & U\\
U \arrow[u, hook] \arrow[ur, hook, "\sim" swap]\arrow[urr, equal, bend right=20]
\end{tikzcd} & \text{and} & 
\begin{tikzcd}
& & V  \\
V \arrow[r] \arrow[urr, equal, bend left=20] \arrow[drr, equal, bend right=20]  & 
V^J \arrow[r, two heads] \arrow[ur, two heads, "\sim"] \arrow[dr, two heads, "\sim" swap] 
& V \times V \arrow[u, two heads, "\pr_1" right] \arrow[d, two heads, "\pr_2" right]\\
& & V,
\end{tikzcd}
\end{tabular}
\]
respectively. By the two-out-of-three property the maps $U\times J\to U$ and $V\to V^J$ are weak equivalences with the latter also a monomorphism, since it is the first of a chain of maps composing to an isomorphism.
\end{proof}

\begin{df}\label{inthtpy}
Let $f,g\colon U\to V$ be maps in $\mathcal{M}$. We say that $f$ and $g$ are \emph{$J$-homotopic} if there is a map $H\colon U \times J\to V$ that we refer to as \emph{$J$-homotopy} such that the following diagram commutes: 
\[
\begin{tikzcd}
U \arrow[d,"\id\times j_0" left] \arrow[rd,"f" above right]\\
U\times J \arrow[r,"H"]  & V,\\
U \arrow[u,"\id\times j_1" left] \arrow[ru,"g" below right] \\
\end{tikzcd}
\]
where $j_0, j_1\colon *\hookrightarrow * \amalg *  \hookrightarrow J$ are the two composites.
By adjointness, a homotopy may equivalently be defined as a map $H'\colon U \to V^{J}$ or $H''\colon J \to V^{U}$ as in the diagrams
\[\begin{tabular}{l c r}
\begin{tikzcd}
& V  \\
U \arrow[r,"H'"] \arrow[ru,"f" above left] \arrow[rd,"g" below left] & V^{J} \arrow[u,"V^{j_0}" right] \arrow[d,"V^{j_1}" right]\\
& V  \\
\end{tikzcd} & \text{and} \quad &
\begin{tikzcd}
* \arrow[d,"j_0" left] \arrow[rd,"{f}" above right]\\
J \arrow[r,"H''"]  & V^{U}.\\
* \arrow[u,"j_1" left] \arrow[ru,"{g}" below right] \\
\end{tikzcd}
\end{tabular}
\]
For additional specificity we may refer to $H$ as a \emph{left homotopy} and $H'$ as a \emph{right homotopy}.
\end{df}

The primary goal of this section is to study the following notions of equivalence for different interval objects.
\begin{df}
We say that a map $f\colon U\to V$ is a \emph{$J$-homotopy equivalence} if there are maps $h,h'\colon V\to U$ and $J$-homotopies  $h f \simeq \id_U$ and $f h' \simeq \id_V$, and we say that $U$ and $V$ are \emph{$J$-homotopy equivalent} if there is a $J$-homotopy equivalence between them.
\end{df}

The following is one of three notions of homotopy equivalence we are interested in.
\begin{ex}
In the category $\sps(C)$ a $\Delta[1]$-homotopy equivalence is precisely a simplicial homotopy equivalence.
\end{ex}

Next we recount some key properties of homotopy equivalences from \cite{ds} adapted for interval objects.

\begin{lemma}\label{htpycomp}
Let $f,g\colon U \to V$ be $J$-homotopic and $h\colon V\to W$ and $h'\colon W'\to U$ any maps. Then $hf,hg\colon U\to W$ are $J$-homotopic, and $fh',gh'\colon W\to V$ are $J$-homotopic.
\end{lemma}

\begin{proof}
Given a $J$-homotopy $H\colon U\times J\to V$ between $f$ and $g$ with an adjoint $H'\colon U\to V^J$, then $hH\colon U\times J\to W$ is a $J$-homotopy between $hf$ and $hg$, and $H'h'\colon W\to V^J$ is a $J$-homotopy between $fh'$ and $gh'$.
\end{proof}

\begin{lemma}\label{twohtpies}
If $J$ and $J'$ are interval objects for $\mathcal{M}$, and $f,g\colon U\to V$ maps with $V$ fibrant, then $f$ and $g$ are $J$-homotopic if and only if they are $J'$-homotopic.
\end{lemma}

\begin{proof}
The interval objects $J$ and $J'$ give rise to the cylinder objects $V^{J}$ and $V^{J'}$ whose factorizations fit into the diagram
\[
\begin{tikzcd}
V \arrow[r, hook, "\sim"] \arrow[d, hook, "\sim" swap] &  V^{J} \arrow[d, two heads]\\
V^{J'} \arrow[r, two heads] \arrow[ur, dashed, "\exists \varphi"]& V\times V.
\end{tikzcd}
\]
Now since the left morphism is an acyclic cofibration and the right morphism a fibration, there is a lift $\varphi\colon  V^{J'}\to V^J$ that is compatible with the factorizations. Now for any right $J'$-homotopy $H\colon U \to V^{J'}$ between $f$ and $g$, the map $\varphi H\colon U \to V^{J}$ is a right $J$-homotopy between the between $f$ and $g$. The argument is symmetric with respect to $J$ and $J'$ proving the claim.
\end{proof}

\begin{lemma}
Let $U,V$ and $W$ be objects in $\mathcal{M}$, and let $f,g\colon U\to V$ be $J$-homotopic. Then $W^f,W^g\colon W^V\to W^U$ are $J$-homotopic.
\end{lemma}

\begin{proof}
If $H\colon U\times J \to V$ is a left $J$-homotopy between $f$ and $g$, then $W^H\colon W^V \to W^{U\times J}\cong (W^U)^ {J}$ is a right $J$-homotopy between $W^f$ and $W^g$.
\end{proof}

The following corollary tells us that homotopy equivalences determined by interval objects are compatible with the Cartesian structure on $\mathcal{M}$.

\begin{cor}\label{cateqcart}
Let $U,V,W\in M$, and let $f\colon  U\to V$ be a $J$-homotopy equivalence. Then $W^f\colon  W^V\to W^U$ is a $J$-homotopy equivalence.
\end{cor}

Note that if $f$ in the above corollary were a weak equivalence, we would generally need $W$ to be fibrant in order to show conclude that $W^f$ is a weak equivalence. This broader compatibility of homotopy equivalence with the Cartesian structure is their main utility to us. Furthermore, the following lemma ties homotopy equivalences to weak equivalences when considered between fibrant-cofibrant objects.

\begin{lemma}\cite[4.24.]{ds}\label{htpywe}
Let $f\colon  U\to V$ be a map in $\mathcal{M}$ with both $U$ and $V$ fibrant-cofibrant. Then $f$ is a $J$-homotopy equivalence if and only if it is a weak equivalence in $\mathcal{M}$.
\end{lemma}

In the canonical model structure on $\cat$, weak equivalences are precisely equivalences of categories and all objects are fibrant-cofibrant, giving us the following result.
\begin{cor}
Two categories are equivalent if and only if they are $I$-homotopy equivalent.
\end{cor}

The following, rather technical result is key for showing that various functors, such as our completion, preserve appropriate notions of equivalence. 

\begin{prop}\label{htpyfun}
Let $\mathcal{M}'$ be another Cartesian model category with monomorphisms included among cofibrations, and let $J$ and $J'$ be interval objects for $\mathcal{M}$ and $\mathcal{M}'$, respectively. Furthermore, let $G\colon \mathcal{M}\to \mathcal{M}'$ be a functor with the following two properties.
\begin{enumerate}
\item The functor $G$ takes the defining diagram for $J$ to the corresponding diagram for $J'$ up to isomorphism:
\[
G\left(
\begin{tikzcd}
*_{M} \arrow[d, hook, "j_0" left] \arrow[dr, equal] \\
 J \arrow[r] & *_{M}\\
*_{M} \arrow[u, hook, , "j_1" left] \arrow[ur, equal]
\end{tikzcd}
\right)
\cong
\begin{tikzcd}
*_{M'} \arrow[d, hook, "j'_0" left] \arrow[dr, equal] \\
 J' \arrow[r] & *_{M'}.\\
*_{M'} \arrow[u, hook, , "j'_1" left] \arrow[ur, equal]
\end{tikzcd}
\]
\item The canonical map $G(U\times J)\to G(U)\times G(J)$ is a $J'$-homotopy equivalence for all $U$.
\end{enumerate}
Then $G$ takes $J$-homotopic maps to $J'$-homotopic maps and $J$-homotopy equivalences to $J'$-homotopy equivalences.
\end{prop}

\begin{proof}
Let $G$ be as in the statement, $\varphi\colon  G(U\times J)\to G(U)\times G(J)$ the canonical map with a weak left inverse $\psi\colon G(U)\times G(J) \to G(U\times J)$ with respect to $J'$-homotopy, and $H\colon U\times J\to V$ a $J$-homotopy between $f,g\colon U\to V$. These maps fit into the following commutative diagram:
\[
\begin{tikzcd}
GU \arrow[drr, "Gf", bend left] \arrow[dr, "G(\id_U\times j_0)"] \arrow[d, "\id_{GU}\times j'_0" swap]  & &\\
GU\times GJ & G(U\times J) \arrow[r, "GH"] \arrow[l, "\varphi" above] &  GV.\\
GU \arrow[urr, "Gg", bend right, swap] \arrow[ur, "G(\id_U\times j_1)" swap] \arrow[u, "\id_{GU}\times j'_1"] & &
\end{tikzcd}
\]
Now since homotopy is compatible with composition by Lemma \ref{htpycomp}, we get the following chain of $J'$-homotopies:
\[
\begin{split}
Gf
& = GH \circ G(\id_U\times j_0)\\
& \simeq GH \circ \psi \circ \varphi \circ G(\id_U\times j_0)\\
& = GH \circ \psi \circ (\id_{GU}\times j'_0)\\
& \simeq GH \circ \psi \circ (\id_{GU}\times j'_1)\\
& = GH \circ \psi \circ \varphi  \circ G(\id_U\times j_1)\\
& \simeq GH \circ G(\id_U\times j_1)\\
& = Gg.
\end{split}
\]
Thus $G$ takes $J$-homotopies to $J'$-homotopies. Then by functoriality of $G$, it also takes $J$-homotopy equivalences to $J'$-homotopy equivalences proving the claim.
\end{proof}

Next we move to our third and most important notion of homotopy. 
\begin{notn}
Let $E=NI$ be the discrete nerve of the walking isomorphism in $\sps(\Delta)$ and recall that it may also be viewed in $\sps(\Theta C)$ via the functor $\pi^*$.
\end{notn}

\begin{df}\label{cateqdef}
We say that two maps in $ \sps(\Theta C)$ are \emph{categorically homotopic} if they are $\pi^* E$-homotopic.
We say a map in $ \sps(\Theta C)$ is a \emph{categorical equivalence} if it is a $\pi^* E$-homotopy equivalence.
\end{df}

\begin{lemma}\label{eint}
The object $\pi^* E$ is an interval object for the $\cpt(S)$-local model structure on $ \sps(\Theta C)$.
\end{lemma}

\begin{proof}
As an element of the set $\cpt$, the map $\pi^* E\to F[0]$ is, in particular, a $\cpt(S)$-local equivalence. The maps $j_0,j_1\colon F[0]\to \pi^* E$ are induced by the inclusions ${{0},{1}\colon *\to I}$ and are monomorphisms since both the discrete nerve and $\pi^*$ preserve monomorphisms as right adjoints.
\end{proof}

\begin{cor}\label{eqcat}
The functor $\pi^* N\colon  \cat \to \sps(\Theta C)$ takes equivalences of categories to categorical equivalences.
\end{cor}

\begin{proof}
The discrete nerve $\pi^*N$ takes the interval object $I$ to the interval object $\pi^*E$ by definition and preserves products as it is a composite of two right adjoints. Thus $\pi^*N$ satisfies the hypotheses of Proposition \ref{htpyfun} and takes $I$-homotopy equivalences to $\pi^*E$-homotopy equivalences.
\end{proof}

Using the facts that $\cpt(S)$-fibrant objects are also injective fibrant and that $\cpt(S)$-local equivalences coincide with levelwise weak equivalences between fibrant objects, we obtain the following corollary using Lemma \ref{htpywe}.

\begin{cor}\label{cateqreedy}
Let $U$ and $V$ be $\cpt(S)$-fibrant objects of $\sps(\Theta C)$ and $f\colon U\to V$. Then the following properties are equivalent for $f$:
\begin{enumerate}
\item categorical equivalence,
\item simplicial homotopy equivalence, and
\item levelwise weak equivalence.
\end{enumerate}
\end{cor}

\begin{proof}
Lemma \ref{htpywe} tells us that between $\cpt(S)$-fibrant objects, categorical equivalences coincide with $\cpt(S)$-local equivalences which further coincide with levelwise weak equivalences giving us the equivalence $(1)\Leftrightarrow (3)$. Since $\cpt(S)$-fibrant objects are, in particular, injective fibrant, Lemma \ref{htpywe} tells us also that levelwise weak equivalences coincide with simplicial homotopy equivalences, giving us the equivalence $(2)\Leftrightarrow (3)$.
\end{proof}

Using the fact that the interval object $\pi^* E$ that defines categorical equivalence is also used to define completeness, we can extend the relationship between homotopy equivalences and weak equivalences from Lemma \ref{htpywe} to more general Segal objects. The following result was shown for $C=*$ by Rezk in \cite[13.6]{rezksp}.
 
\begin{prop}\label{cateqcpt}
If a map between $\se(S)$-fibrant objects of $\sps(\Theta C)$ is a categorical equivalence, then it is a $\cpt(S)$-local equivalence.
\end{prop}
 
\begin{proof}
Let $X,Y,Z\in \sps(\Theta C)$ with $X$ and $Y$ $\se(S)$-fibrant, and $Z$ $\cpt(S)$-fibrant, and let $f\colon X\to Y$ be a categorical equivalence. Now by Corollary \ref{cateqcart} $Z^f$ is a categorical equivalence between the $\cpt(S)$-fibrant objects $Z^Y$ and $Z^X$. Thus $Z^f$ is a levelwise equivalence by Corollary \ref{cateqreedy}, in particular at level $0$. Thus the map
\[
\Map(f,Z)\cong \Map(F[0],Z^{f}) \cong (W^{f})_0
\]
is a weak equivalence of spaces, which precisely means that $f$ is a $\cpt(S)$-local equivalence.
\end{proof}

\section{DK-equivalences}\label{sdk}

In this section we study DK-equivalences, short for Dwyer-Kan equivalences, which may be viewed as maps that induce equivalences of certain \emph{enriched homotopy categories} and behave like essentially surjective and fully faithful functors. We establish connections between DK-equivalences and levelwise and categorical equivalences. DK-equivalences for Segal spaces were extensively studied by Rezk in \cite{rezksp}, and both our results and some of our techniques generalize those found there. Similar results regarding DK-equivalence have been shown for $\Theta_n$-spaces in slightly more restrictive form also by Bergner in \cite{bdisc} using alternate methods.

We start with a few results regarding homotopy categories which are used to define the essential surjectivity part of DK-equivalences.

\begin{df}
Let $X$ be a $\se(S)$-fibrant object of $\sps(\Theta C)$. Xe define the \emph{homotopy category} of $X$, denoted by $\ho(X)$, as the category with objects $X_{[0],0}$ and $\Hom$-sets given by
\[
\Hom_{\ho(X)}(x,y)= \pi_0 (\map_X(x,y)_t))
\]
for objects $x$ and $y$. Recall that $t$ is a terminal object of $C$.
\end{df}

The homotopy category construction is functorial and provides a one-sided inverse to the discrete nerve, giving us the following converse to Corollary \ref{eqcat}.
\begin{lemma}\label{cateqhocat}
If two $\se(S)$-fibrant objects of $\sps(\Theta C)$ are categorically equivalent, then they have equivalent homotopy categories.
\end{lemma}

\begin{proof}
First note that $\ho\pi^*N\cong \id_{\cat}$ and in particular, $\ho(\pi^*E)\cong I$. The functor $\ho$ also preserves products, since $\pi_0$ and evaluation at level $([0],0)$ as maps $\sps(\Theta C)\to \set$ 
do so. Thus the functor $\ho\colon \sps(\Theta C)\to \cat$ satisfies the hypotheses of Proposition \ref{htpyfun}, which tells us that $\ho$ takes categorical equivalences to equivalences of categories. 
\end{proof}

We recall the following characterization of isomorphisms in the homotopy category from \cite[5.5]{rezksp}.
Given a Segal object $X$, a map $\varphi\in \map_X(x,y)_{[0],0}$ represents an isomorphism in $\ho(X)$ if and only there are maps $\psi,\psi'\in \map_X(y,x)_{[0],0}$ such that $\varphi\psi$ and $\psi'\varphi$ are homotopic to the identities on $y$ and $x$, respectively. We call such $\varphi$ a \emph{homotopy equivalence} in $X$, and they are closed under homotopy by \cite[5.8]{rezksp}. As a consequence we may consider the space $X_{eq}\subseteq X_{[1](t)}$ consisting of their components. 

The interval object $\pi^* E$ corepresents a notion of equivalence that is a priori stronger, where the two weak inverses $f$ and $h$ to $g$ are required to coincide. However, the following proposition says that every homotopy equivalence is homotopic to one with equal left and right weak inverse.

\begin{prop}\cite[7.8]{rezkth}\label{eeq}
Let $X$ be a $\se(S)$-fibrant object in $\sps(\Theta C)$. Then the map
\[
\Map(\pi^* E, X)\to \Map( \pi^* F[1],X)\cong (\tau^*X)_1
\]
induced by either inclusion $[1]\hookrightarrow I$ factors through $X_{eq}\subseteq (\tau^*X)_1 $ and induces a weak equivalence ${\Map(\pi^* E, X)\to X_{eq}}$ of spaces.
\end{prop}

For a Segal object $X$ the set of isomorphism classes of objects of $\ho(X)$ coincides with the set of homotopy equivalence classes of objects of $X$. However, when $X$ is complete, two of its objects homotopy equivalent precisely when there is a path connecting them by Proposition \ref{eeq}, which implies that the set of homotopy equivalence classes of objects may be computed as $\pi_0(X_{[0]})$. For a general $\se(S)$-fibrant $X$ the set of homotopy equivalence classes of objects is a quotient of $\pi_0(X_{[0]})$.

\begin{df}\label{dkdef}
Let $f\colon X\to Y$ be a map between $\se(S)$-fibrant objects of $\sps(\Theta C)$. We say that $f$ is a \emph{DK-equivalence} if 
\begin{enumerate}
\item the functor $\ho(f)\colon \ho(X)\to \ho(Y)$ is an equivalence of categories, and

\item the maps $\map_X(x,y)\to \map_Y(f(x),f(y))$ are levelwise weak equivalences for all $x,y\in X_{[0],0}$.
\end{enumerate}
\end{df}

The second condition, which we refer to as $f$ being \emph{homotopically fully faithful}, implies that $f$ is fully faithful on homotopy categories, so the first condition is may be reduced to
\begin{enumerate}
\item[(1')] the functor $\ho(f)\colon \ho(X)\to \ho(Y)$ is essentially surjective.
\end{enumerate}
Furthermore, since essential surjectivity is equivalent to surjectivity on isomorphism classes of objects, condition (1') can be reduced to the following when $Y$ is $\cpt(S)$-fibrant:
\begin{enumerate}
\item[(1'')] the map $\pi_0(f)\colon \pi_0(X_{[0]})\to \pi_0(Y_{[0]})$ is surjective.
\end{enumerate}

While this notion of a DK-equivalence captures the correct notion of equivalence for enriched categories, when higher-dimensional completeness conditions are omitted in the $n$-categorical setting, a weaker notion of DK-equivalence is required. We provide the adapted notion for the $n$-dimensional setting in Definition \ref{dkn}.

\begin{df}
Let $\mathcal{W}$ be a class of morphisms in a category $\mathcal{C}$. We say that $\mathcal{W}$ satisfies the \emph{two-out-of-six} property if for any sequence of morphisms
\[
\begin{tikzcd}[sep=small]
X \arrow[r, "f"]
& Y \arrow[r, "g"]
& Z \arrow[r, "h"]
& W
\end{tikzcd}
\]
if $fg$ and $gh$ are in $\mathcal{W}$, then so are $f,g,h$ and $fgh$.
\end{df}
The two-out-of-six along with the two-out-of-three property is a useful tool for obtaining new equivalences from old ones. As a classical example, weak equivalences in any model category satisfy the two-out-of-six property, and since both levelwise weak equivalences and equivalences of categories have the two-out-of-six and two-out-of-three properties, so do DK-equivalences.

Next we examine the relationship of DK-equivalences to levelwise weak equivalences.

\begin{lemma}\label{lwdk}
Let $f$ be a map between $\se(S)$-fibrant objects of $\sps(\Theta C)$. Then $f$ is a levelwise weak equivalence if and only if it is a DK-equivalence and a weak equivalence at level $[0]$.
\end{lemma}

\begin{proof}
Let $f\colon X\to Y$ be a map of $\se(S)$-fibrant objects such that $f_{[0]}$ is a weak equivalence. 

Then consider the diagram
\begin{equation}\label{lwdkd}
\begin{tikzcd}
X_{[1](c)}\arrow[r,"f_{[1](c)}"] \arrow[d, two heads, "(d_1{,}d_0)" left]
& Y_{[1](c)} \arrow[d, two heads, "(d_1{,}d_0)"]\\
X_{[0]}\times X_{[0]} \arrow[r,"f_{[0]}\times f_{[0]}" above, "\sim" below] 
& Y_{[0]}\times Y_{[0]},
\end{tikzcd}
\end{equation}
where $f_{[0]}\times f_{[0]}$ is a weak equivalence, since the Kan-Quillen model structure is Cartesian. Furthermore, the objects $X$ and $Y$ being injective fibrant implies that the vertical arrows are fibrations, so the fibers are weakly equivalent to the homotopy fibers. Thus the induced map on the homotopy fibers is
\[
\begin{tikzcd}
\map_X(x,y)_c \arrow[r, "f_*"]
& \map_Y(f(x),f(y))_c.
\end{tikzcd}
\]
Now $f$ is a levelwise weak equivalence on the mapping objects if and only if the squares (\ref{lwdkd}) are homotopy pullbacks for all $c\in C$, which happens if and only if $f_{[1](c)}$ is a weak equivalence.
Now if we suppose that $f$ is a DK-equivalence, then $f$ is a weak equivalence on levels $[0]$ and $[1](c)$ for all $c$ which is sufficient for $f$ to be a levelwise equivalence by Proposition \ref{segallw}.

Conversely, if $f$ is a levelwise weak equivalence then it induces levelwise weak equivalences on all mapping objects and is a simplicial homotopy equivalence, the latter of which implies that it induces an equivalence of homotopy categories. Thus $f$ is a DK-equivalence.
\end{proof}

Between Segal objects levelwise weak equivalences are stronger than DK-equivalences, since the essential surjectivity condition is too weak to control the behaviour of the level $[0]$. However, the completeness condition allows us to control the level $[0]$ in part by using information on 1-cells. To establish this connection, we use the following notion as a tool.

\begin{df}
A map of Kan complexes is a \emph{homotopy monomorphism} if it is injective on $\pi_0$ and a weak equivalence restricted to each component.
\end{df}

Note that a component-wise weak equivalence that is bijective on $\pi_0$ is a weak equivalence.
\begin{lemma}\label{sseteq}
A map of simplicial sets is a weak equivalence if and only if it is a homotopy monomorphism and surjective on $\pi_0$.
\end{lemma}

The following lemma allows us to characterize homotopy monomorphisms in terms of homotopy limits.

\begin{lemma}\label{htpymono}
Let $f\colon X\to Y$ be a map of Kan complexes. Then $f$ is a homotopy monomorphism if and only if the following diagram is a homotopy pullback:
\begin{equation}\label{hmonodiag}
\begin{tikzcd}
X \arrow[d, "\Delta" left] \arrow[r, "f" above] & Y \arrow[d, "\Delta" ]  \\
X\times X \arrow[r, "f\times f"] & Y \times Y.\\
\end{tikzcd}
\end{equation}
\end{lemma}

\begin{proof}
The diagram (\ref{hmonodiag}) is naturally weakly equivalent to 
\[
\begin{tikzcd}
\Map(\Delta[1],X) \arrow[d, two heads] \arrow[r, "f_*" above] 
& \Map(\Delta[1],Y) \arrow[d, two heads]  \\
\Map(\Delta[0]\amalg \Delta[0],X) \arrow[r, "f_*"] 
& \Map(\Delta[0]\amalg \Delta[0],Y), \\
\end{tikzcd}
\]
where the vertical arrows are Kan fibrations since $X$ and $Y$ are Kan complexes. The objects on the top row are spaces of free paths in $X$ and $Y$, respectively, and thus the fibers of the vertical morphisms are spaces of path with fixed endpoints. Furthermore, two points $x_0,x_1\in X_0$ are in the same same component if and only if the space of paths between them is non-empty, and similarly for $f(x_0),f(x_1)$ in $Y_0$. Additionally, for $x_0$ and $x_1$ in the same path component, the space of paths between them is weakly equivalent to the based loop space on either of the points.

Thus the diagram (\ref{hmonodiag}) is a homotopy pullback if and only if $f$ induces weak equivalences on all the path spaces with fixed endpoints, which in turn is equivalent to $f$ being injective on $\pi_0$ and inducing weak equivalences on the loop spaces of all components of $X$ and $Y$. However, a weak equivalence of loop spaces is equivalently a map that induces isomorphisms on $\pi_i$ for $i\geq 1$. Since each component is also connected, $f$ is a weak equivalences on the loop spaces of all components if and only if it is a weak equivalence on all components. Thus the diagram (\ref{hmonodiag}) is a homotopy pullback if and only if $f$ is a homotopy monomorphism.
\end{proof}

Now we are ready to prove that between complete Segal objects DK-equivalences coincide with levelwise weak equivalences and thus also with $\cpt(S)$-local equivalences, categorical equivalences and simplicial homotopy equivalences.

\begin{prop}\label{dklw}
A map between $\cpt(S)$-fibrant objects in $\sps(\Theta C)$ is a levelwise weak equivalence if and only if it is a DK-equivalence.
\end{prop}

\begin{proof}
Due to Lemma \ref{lwdk} we only need to prove the reverse direction, where it suffices to show that $f_{[0]}$ is a weak equivalence.
Let $f\colon X\to Y$ be a DK-equivalence of $\cpt(S)$-fibrant objects. Since $f$ induces levelwise weak equivalences on all mapping spaces, the following diagram is a homotopy pullback square for any $c\in C$:
\[
\begin{tikzcd}
X_{[1](c)} \arrow[r,"f_{[1](c)}" above] \arrow[d, two heads, "(d_1{,}d_0)" left] \arrow[dr, phantom, "\lrcorner_h" very near start] & Y_{[1](c)} \arrow[d, two heads, "(d_1{,}d_0)" right]\\
X_{[0]}\times X_{[0]} \arrow[r,"f_{[0]}\times f_{[0]}" above] & Y_{[0]}\times Y_{[0]}.
\end{tikzcd}
\]
Restricting the case $c=t$ to the components consisting of homotopy equivalences, we get the following diagram:
\[
\begin{tikzcd}
X_{[0]} \arrow[r,"f_{[0]}" above] \arrow[d, "s_0" left , "\sim" right] & Y_{[0]} \arrow[d, "s_0" left, "\sim" right ]\\
X_{eq} \arrow[r,"f_{eq}" above] \arrow[d, two heads] \arrow[dr, phantom, "\lrcorner_h" very near start] & Y_{eq} \arrow[d, two heads]\\
X_{[0]}\times X_{[0]} \arrow[r,"f_{[0]}\times f_{[0]}" above] & Y_{[0]}\times Y_{[0]}
\end{tikzcd}
\]
where the maps $s_0$ are weak equivalences, since $X$ and $Y$ are $\cpt(S)$-fibrant. It follows that the outer rectangle is also a homotopy pullback with diagonals as the vertical maps:
\[
\begin{tikzcd}
X_{[0]} \arrow[r,"f_{[0]}" above] \arrow[d, "\Delta" left] \arrow[dr, phantom, "\lrcorner_h" very near start]  & Y_{[0]} \arrow[d, "\Delta"]\\
X_{[0]}\times X_{[0]} \arrow[r,"f_{[0]}\times f_{[0]}" above] & Y_{[0]}\times Y_{[0]}.
\end{tikzcd}
\]
Since this diagram is a homotopy pullback, $f_{[0]}$ is a homotopy monomorphism by Lemma \ref{htpymono}. Due to the completeness of $Y$ the essential surjectivity of $f$ on the homotopy categories is equivalent to surjectivity on $\pi_0$. 
Thus $f_{[0]}$ is both a homotopy monomorphism and surjective on $\pi_0$ and hence a weak equivalence by Lemma \ref{sseteq}.
\end{proof}

Our next goal is to establish a connection between DK-equivalences and categorical equivalences also for Segal objects without completeness. To this end, we require a few technical lemmata.

\begin{lemma}\label{pullbacks}
Let $X$ be $\se(S)$-fibrant object in $\sps(\Theta C)$. Then the following diagrams are (homotopy) pullback squares:
\[
\begin{tikzcd}
X_{[1](c_1)} \arrow[r, "s_1"] 
\arrow[d, "d_0" left] \arrow[dr, phantom, "\lrcorner" very near start] & X_{[2](c_1,c_2)} \arrow[d, "d_0" right] & 
X_{[1](c_2)} \arrow[r, "s_0"] 
\arrow[d, "d_1" left] \arrow[dr, phantom, "\lrcorner" very near start] & X_{[2](c_1,c_2)} \arrow[d, "d_2" right]\\
X_{[0]} \arrow[r,"s_0" above] & X_{[1](c_2)} &
X_{[0]} \arrow[r,"s_0" above] & X_{[1](c_1)}.
\end{tikzcd}
\]
\end{lemma}

\begin{proof}
In the left-hand diagram we can use the Segal condition to replace $X_{[2](c_1,c_2)}$ by ${X_{[1](c_1)}\tim_{X_{[0]}}X_{[1](c_2)}}$, and then perform the cancellations 
\[\begin{split}
\left(X_{[1](c_1)}\tim_{X_{[0]}}X_{[1](c_2)}\right)\tim_{X_{[1](c_2)}}X_{[0]} & \cong X_{[1](c_1)}\tim_{X_{[0]}}X_{[0]}\\
& \cong X_{[1](c_1)}.
\end{split}\] 
Furthermore, the squares are also homotopy pullbacks due to injective fibrancy of $X$. The argument for the right diagram is similar.
\end{proof}

The following lemma identifies a product decomposition for the suspension functor.

\begin{lemma}\cite[4.9]{rezkth}\label{itprod}
Let $X$ and $Y$ be objects of $\sps(C)$. Then the following diagram is a pushout square:
\[
\begin{tikzcd}[sep=huge]
V[1](X\times Y) \arrow[d, "V{[}\delta^1{]}(\id_{X\times Y})" left] \arrow[r, "V{[}\delta^1{]}(\id_{X\times Y})" above] \arrow[dr, phantom, "\ulcorner" very near end]& V[2](X,Y) \arrow[d, "(V{[}\sigma^1{]}(\id_X){,}V{[}\sigma^0{]}(\id_Y))" right ]\\
 V[2](Y,X) \arrow[r, "(V{[}\sigma^0{]}(\id_X){,}V{[}\sigma^1{]}(\id_Y))" {below, yshift=-5pt}]&   V[1](X)\times V[1](Y),
\end{tikzcd}
\]
where $\delta^i$ and $\sigma^i$ denote the $i\si{th}$ face and degeneracy maps, respectively.
\end{lemma}

Note that the inclusion of components $X_{eq}\subseteq X_{[1](t)}$ is, in particular, a homotopy monomorphism, which combined with the weak equivalence described in Proposition \ref{eeq} gives us the following result.
\begin{lemma}\label{heqmono}
Let $X$ be a $\se(S)$-fibrant object in $\sps(\Theta C)$. Then the maps
\[
\Map(\pi^* E, X)\to \Map( \pi^* F[1],X)
\]
induced by the inclusions $[1]\hookrightarrow I$ are homotopy monomorphisms.
\end{lemma}

The next lemma tells us that the projections associated to the path objects for the interval $\pi^* E$ are DK-equivalences, creating a link to categorical equivalences.

\begin{lemma}\label{edk}
Let $X$ be a $\se(S)$-fibrant object of $\sps(\Theta C)$. Then the maps $r^*\colon  X\to X^{\pi^* E}$ and $(j_0)^*,(j_1)^*\colon  X^{\pi^* E}\to X$ induced by $r\colon \pi^* E\to F[0]$ and the inclusions $j_0,j_1\colon  F[0]\to E$ are DK-equivalences.
\end{lemma}

\begin{proof}
By the two-out-of-three property it suffices to show that $r^*$ is a DK-equivalence.
The map $r$ is given by a discrete nerve of an equivalence of categories and is thus a categorical equivalence, so Corollary \ref{cateqcart} tells us that the induced map $r^*\colon  X\to X^{\pi^* E}$ is also a categorical equivalence. Then the induced functor on homotopy categories $\ho(X)\to \ho(X^{\pi^* E})$ is an equivalence by Lemma \ref{cateqhocat}. We want to show that we have levelwise weak equivalences on mapping objects $\map_X(x,y)\to \map_{X^{\pi^* E}}(r^*x,r^*y)$, which we do via the two-out-of-three property.

Consider the inclusion $i\colon F([1](t))\cong \pi^*F[1]\hookrightarrow \pi^* E$ and the map ${i^*\colon X^{\pi^* E}\to X^{F([1](t))}}$ it induces, which appears in the diagram
\begin{equation}\label{eqhtpymono}
\begin{tikzcd}
X_{[1](c)} \arrow[r,"r^*_{[1](c)}" above] \arrow[d, two heads] & (X^{\pi^* E})_{[1](c)} \arrow[r,"i^*_{[1](c)}" above] \arrow[d, two heads] & (X^{F([1](t))})_{[1](c)} \arrow[d, two heads]\\
X_{[0]}\times X_{[0]} \arrow[r,"r^*_{[0]}\times r^*_{[0]}" above] & (X^{\pi^* E})_{[0]}\times (X^{\pi^* E})_{[0]} \arrow[r,"i^*_{[0]}\times i^*_{[0]}" above] & (X^{F([1](t))})_{[0]}\times (X^{F([1](t))})_{[0]}.
\end{tikzcd}
\end{equation}
Here the composites of the horizontal maps are given by the degeneracy $i\circ r = \sigma^0\colon [1](t)\to [0]$, and the maps induced by $i^*$ are homotopy monomorphisms by Lemma \ref{heqmono}, since via natural isomorphisms they correspond to the maps
\[
\begin{tikzcd}[cramped, sep=small]
\Map( \pi^* E,  X^{F([1](c))}) \arrow[r, "i^*"] & \Map( \pi^*F[1], X^{F([1](c))})
\end{tikzcd}
\]
and
\[
\begin{tikzcd}[cramped, sep=small]
\Map( \pi^* E, X^{F[0]\amalg F[0]}) \arrow[r, "i^*"] & \Map(\pi^*F[1],  X^{F[0]\amalg F[0]}).
\end{tikzcd}
\]
By commutativity of homotopy limits, we also have levelwise homotopy monomorphisms on the homotopy fibers of the vertical maps of the right square in diagram (\ref{eqhtpymono}),
\[
\map_{X^{\pi^* E}}(x,y)\to \map_{X^{F([1](t))}}(i^*x,i^*y).
\]
Next we show that the large rectangle in the diagram (\ref{eqhtpymono}) above is a homotopy pullback for each $c$.
To this end, consider the isomorphisms
\[\begin{split}
X_{[2](t,c)}\tim_{X_{[1](c)}} X_{[2](c,t)}
&\cong \Map(F([2](t,c))\amg_{F([1](c))} F([2](c,t)), X) \\
&\cong \Map(F([1](c))\tim F([1](c)), X)\\
&\cong(X^{F([1](c))})_{[1](c)},
\end{split}\]
the second of which follows from Lemma \ref{itprod}. 
Now the outer rectangle of (\ref{eqhtpymono}) is equivalent to
\[
\begin{tikzcd}
X_{[1](c)} \arrow[r, "(s_0{,}s_1)" above] \arrow[d, "(d_1{,}d_0)" left , two heads] & X_{[2](t,c)}\times_{X_{[1](c)}} X_{[2](c,t)} \arrow[d, "(d_2\pr_1{,}d_0 \pr_2)" right, two heads]\\
X_{[0]}\times X_{[0]} \arrow[r,"s_0\times s_0" above] & X_{[1](t)}\times X_{[1](t)}.
\end{tikzcd}
\]
Then, by viewing the products on the bottom row as fibered over the terminal object $\Delta[0]$, we can use commutativity of homotopy pullbacks with one another to decompose the diagram into the three other diagrams below, which are seen to be homotopy pullbacks in Lemma \ref{pullbacks}:
\[
\begin{tikzcd}
X_{[1](c)} \arrow[r, "(s_0{,}s_1)" above] \arrow[d, "(d_1{,}d_0)" left , two heads] & X_{[2](t,c)}\times_{X_{[1](c)}}^h X_{[2](c,t)} \arrow[d, "(d_2\pr_1{,}d_0 \pr_2)" right, two heads]  &
X_{[1](c)} \arrow[r, "s_1"] 
\arrow[d, "d_0" left, two heads] \arrow[dr, phantom, "\lrcorner_h" very near start] & X_{[2](c,t)} \arrow[d, "d_0" right, two heads] \\
X_{[0]}\times X_{[0]} \arrow[r,"s_0\times s_0" above] & X_{[1](t)}\times X_{[1](t)}  &
X_{[0]} \arrow[r,"s_0" above] & X_{[1](t)}\\
X_{[1](c)} \arrow[r, "s_0"] 
\arrow[d, "d_1" left, two heads] \arrow[dr, phantom, "\lrcorner_h" very near start] & X_{[2](t,c)} \arrow[d, "d_2" right, two heads] &
X_{[1](c)} \arrow[r, equal] \arrow[d, two heads] \arrow[dr, phantom, "\lrcorner_h" very near start] & X_{[1](c)} \arrow[d, two heads]\\
X_{[0]} \arrow[r,"s_0" above] & X_{[1](t)} &
\{*\} \arrow[r, equal] & \{*\}.
\end{tikzcd}
\]
Since the homotopy pullback of the three homotopy pullbacks is precisely what appears in the top left corner of the top left square, it is a homotopy pullback. Thus the outer rectangle of diagram (\ref{eqhtpymono}) is also a homotopy pullback  square, and hence the induced map on the homotopy fibers is a weak equivalence.

Now we have the following commutative diagram on the mapping spaces:
\[
\begin{tikzcd}
\map_X(x,y)_c \arrow[r, "(s_0)_*" above, "\sim" below] \arrow[d, "(r^*)_*" left] &\map_{X^{F([1](t))}}(s_0 x,s_0 y)_c\\
\map_{X^{\pi^* E}}(r^* x,r^* y)_c \arrow[ur, "(i^*)_*" below right]
\end{tikzcd}.
\]
In particular, $(s_0)_*=(i^*)_*(r^*)_*$ is surjective on $\pi_0$ and hence so is $(i^*)_*$ which then is a weak equivalence by Lemma \ref{sseteq} as it is also a homotopy monomorphism. By the two-out-of-three property, $r^*$ also induces weak equivalences on mapping spaces levelwise and is thus a DK-equivalence.

\end{proof}

\begin{prop}\label{cateqdk}
If a map between $\se(S)$-fibrant objects of $\sps(\Theta C)$ is a categorical equivalence, then it is a DK-equivalence.
\end{prop}

\begin{proof}
Let $f\colon X\to Y$ be a categorical equivalence between $\se(S)$-fibrant objects with left and right weak inverses $g,h\colon Y\to X$, respectively. Consider the categorical homotopies $H,H'$ as in the diagrams below:

\[
\begin{tikzcd}
& & X & & & Y \\
& Y \arrow[ru, "h" above left] &  &  & X \arrow[ru, "f" above left] & \\
X \arrow[ru, "f" above left] \arrow[rr, "H" ] \arrow[rrdd, equal] & & X^{\pi^*E} \arrow[uu, "(j_0)^*" right] \arrow[dd, "(j_1)^*" right] & Y \arrow[ru, "g" above left] \arrow[rr, "H'" ] \arrow[rrdd, equal] &  & Y^{\pi^*E} \arrow[uu, "(j_0)^*" right] \arrow[dd, "(j_1)^*" right] \\
& & & & & \\
& & X & & & Y.                     
\end{tikzcd}
\]
By Lemma \ref{edk} the maps $(j_0)^*,(j_1)^*$ are DK-equivalences, as are the identities. Then by the two-out-of-three property, $fh$ and $gf$ are also DK-equivalences, and by the two-out-of-six property so is $f$.
\end{proof}

\section{Horizontal completion}\label{shcomp}

In this section we consider the completion in the lowest dimension, which serves as an inductive base case for the more general completion constructions in Section \ref{sscomp}.

We recall the following definition for our construction.

\begin{df}
Let ${\mathbf{X}}$ be a simplicial object in $\sps(C)$. We define the \emph{simplicial diagonal} of ${\mathbf{X}}$ as the object of $\sps(C)$ given levelwise as 
\[
\diag_q({\mathbf{X}}_{q})_{c,p}:=({\mathbf{X}}_{p})_{c,p}.
\]
\end{df}

Given a Segal object $X$ in $\sps(\Theta C)$, a failure of completeness indicates that $X$ has homotopy equivalences that are not recognized by homotopy-theoretic structure, which may be seen as the space $X_{eq}$ having more components than $X_{[0]}$. To address this incongruence, we may encode the homotopy equivalences in a simplicial resolution of $X$ which we may then geometrically realize to obtain an object equivalent to the original and homotopy equivalences encoded also in the homotopical structure.

A starting point for our resolution is extend the interval object $\pi^*E$ into a cosimplicial object which we may then cotensor with to obtain the simplicial resolution. Let $I(p)$ be the category consisting of a chain of $p$ composable isomorphisms or equivalently the connected groupoid on $p+1$ objects or the groupoid completion of $[p]$. Then the discrete nerves $\pi^* E(p):= \pi^* N I(p)$ assemble into a cosimplicial object in $\sps(\Theta C)$ with $\pi^* E(0)=\pi^* N (*)\cong F[0]$.

\begin{con}\label{completion}
Consider the functor on $\sps(\Theta C)$ that sends an object $X$ to the simplicial object $[p]\mapsto \mathbf{T}_{p}^1 X := X^{\pi^* E(p)}$. We define the \emph{dimension $1$ precompletion} functor $\widetilde{T}^1\colon \sps(\Theta C)\to \sps(\Theta C)$ as the simplicial diagonal of $\mathbf{T}^1$, which may be described levelwise as 
\[
(\widetilde{T}^1 X)_{\theta,p}=(X^{\pi^* E(p)})_{\theta,p}\cong\Hom(F(\theta)\times \pi^* E(p)\times \Delta[p],X)
\]
for $\theta\in \Theta C$ and $[p]\in \Delta$.
In general, the simplicial diagonal may break injective fibrancy, so we define the \emph{dimension $1$ completion} functor $T^1$ as the composite $\mathcal{F} \widetilde{T}^1$ for a for a fixed functorial injective fibrant replacement $\mathcal{F}$.

Additionally, the morphisms $\pi^* E(p)\to F[0]$ induce natural maps $X\cong X^{F[0]}\to X^{\pi^* E(p)}$ whose diagonal followed by the fibrant replacement maps define a natural transformation $\eta^1\colon \id\Rightarrow T^1$.
\end{con}

While the functor $T^1$ is well defined for any presheaf in $\sps(\Theta C)$, we are only interested in its behaviour on Segal objects, since only then is the completeness condition meaningful.

In the case $C=*$, the functor $T^1$ coincides with the Rezk completion \cite[\S 14]{rezksp}, and for $C=\Theta_{n-1}$, functor $T^1$ may also be viewed as an extension of the simplicial completion of $n$-quasi-categories as we now explain.

\begin{ex}
Consider $C=\Theta_{n-1}$ and let $\iota\colon\ps(\Theta_n)\hookrightarrow \sps(\Theta_n)$ be the discrete inclusion. Then $\widetilde{T}^1\iota$ is naturally isomorphic to the right adjoint of the Quillen equivalence between $n$-quasi-categories and complete Segal $\Theta_n$-spaces associated to simplicial completion as described in \cite[8.8]{ara}. Considering the same adjunction with respect to model structures without completeness in dimensions above 1 shows that $\widetilde{T}^1 X$ is already injective fibrant for any $X$ in the image of fibrant objects under the right adjoint, which includes all discrete Segal $\Theta_n$-spaces.

The injective fibrant replacement not being required in the above case is significant as $\widetilde{T}^1$ is a right adjoint and completely explicit in construction, whereas this is not the case for a generic injective fibrant replacement functor $\mathcal{F}$.
\end{ex}

These examples already show that the completions of certain objects are indeed complete, and we show this to be the case for any Segal object as a part of the main result of this section, Theorem \ref{hcompletion}.

We start by examining how the completion behaves with respect to categorical equivalences by considering the completion of the interval object $\pi^* E$. We may reduce the computation to the case of Segal spaces, where the precompletion of the discrete nerve is known to coincide with the homotopy coherent nerve functor called the classifying diagram as discussed in \cite[14.2]{rezksp}. In particular, the classifying diagram preserves weak equivalences, giving us the following result.

\begin{lemma}\cite[3.5]{rezksp}
When $C$ is the terminal category $*$, the functor $\widetilde{T}^1 N\colon \cat \to \sps(\Theta C)\cong \sps(\Delta)$ takes equivalences of categories to $\cpt$-local equivalences.
\end{lemma}

The (pre)completion also commutes with inclusions induced by enlarging the category $C$, which the next lemma shows in the case when starting with simplicial spaces.

\begin{lemma}
Let $X\in \sps(\Delta)$ be a simplicial space. Then $\widetilde{T}^1\pi^* X\cong \pi^*\widetilde{T}^1 X$.
\end{lemma}

\begin{proof}
Let $\theta\in \Theta C$ and $[p]\in \Delta$, and recall that $\pi^*\cong \tau_{!}$. Then using the adjunction $\pi_! \dashv \pi^*$ we obtain natural isomorphisms:
\[
\begin{split}
(\widetilde{T}^1\pi^* X)_{\theta,p} 
& \cong \Hom_{\sps(\Theta C)}(F(\theta)\times \pi^* E(p)\times \Delta[p], \pi^* X)\\
& \cong \Hom_{\sps(\Delta)}(\pi_!(F(\theta)\times \tau_! E(p)\times \Delta[p]), X).\\
\end{split}
\]
We may then observe that the functor $\pi_!$ commutes with finite products, since the product as a left adjoint preserves left Kan extensions; see for example \cite[6.3.2.]{riehl} for this fact. Additionally, on the representables we have the formula $\pi_! F\cong F \pi$, and $\pi_!$ acts as the identity on constant presheaves giving us further natural isomorphisms:
\[
\begin{split}
(\widetilde{T}^1\pi^* X)_{\theta,p} 
& \cong \Hom_{\sps(\Delta)}(\pi_!F(\theta)\times \pi_!\tau_! E(p)\times \Delta[p], X)\\
& \cong \Hom_{\sps(\Delta)}(F(\pi\theta)\times (\pi\tau)_! E(p)\times \Delta[p], X)\\
& \cong \Hom_{\sps(\Delta)}(F(\pi\theta)\times E(p)\times \Delta[p], X)\\
& = (\widetilde{T}^1 X)_{\pi\theta,p}\\
& = (\pi^*\widetilde{T}^1 X)_{\theta,p}.
\end{split}
\]
\end{proof}

\begin{cor}
The unique map $ T^1 \pi^* E\to F[0]$ is a levelwise weak equivalence.
\end{cor}

\begin{proof}
We utilize the two-out-of-three property using the commutative diagram
\[
\begin{tikzcd}
\pi^*\widetilde{T}^1 E \arrow[d, "\pi^*r"  swap]  \arrow[r, leftrightarrow, "\cong"] 
& \widetilde{T}^1\pi^* E  \arrow[r, "\sim"] 
& T^1 \pi^* E  \arrow[d]\\ 
\pi^*F[0] \arrow[rr, leftrightarrow, "\cong"] & & F[0],
\end{tikzcd}
\]
where $r\colon \widetilde{T}^1 E\to F[0]$ is the unique map to the terminal object, which is a levelwise weak equivalence, since it is the Rezk nerve applied to the equivalence $I\to [0]$. The map $\pi^*r$ is then also a levelwise weak equivalence, since $\pi^*$ preserves levelwise weak equivalences as a precomposition functor. The isomorphisms and the fibrant replacement are levelwise weak equivalences, so then is the remaining vertical map on the right by the two-out-of-three property as claimed.
\end{proof}

Then since the inclusions $F[0]\to T^1\pi^* E$ are cofibrations, we obtain the following key results.
\begin{cor}
The object $T^1\pi^* E$ is an interval object for the injective model structure on $\sps(\Theta C)$.
\end{cor}

\begin{cor}
The completion functor takes categorical equivalences to levelwise weak equivalences.
\end{cor}

\begin{proof}
The functor $\widetilde{T}^1$ preserves products, so $T^1$ preserves them up to levelwise weak equivalence or equivalently up to $T^1\pi^* E$-homotopy equivalence by Lemma \ref{twohtpies}. Thus $\widetilde{T}^1$ takes $\pi^* E$-homotopy equivalences to $T^1\pi^* E$-homotopy equivalence by Proposition \ref{htpyfun} proving the claim.
\end{proof}

In order to study the behaviour of the completion we also need some properties of the simplicial diagonal, which may also be understood as a geometric realization. The following theorem, which is originally due to Bousfield and Kan \cite[Ch. XI, 2.6]{bk}, characterizes the diagonal as a homotopy colimit.

\begin{thm}\cite[18.7.4., 15.11.6.]{hirsch} \label{bousfieldkan}
Let ${\mathbf{X}}$ be an injective cofibrant simplicial object in a simplicial model structure on $\sps(C)$. 
Then there is a natural weak equivalence 
\[
\hocolim_{q\in \Delta} {\mathbf{X}}_q \overset{\sim}{\to} \diag_q {\mathbf X}_q
\]
called the Bousfield-Kan map.
\end{thm}

In particular, the diagonal commutes with other homotopy colimits and is homotopy invariant.

\begin{cor}\label{bousfieldkanwe}
Let $f\colon{\mathbf{X}}\to{\mathbf{Y}}$ be a map between injective cofibrant simplicial objects in a simplicial model structure on $\sps(C)$ such that $f_q\colon {\mathbf{X}}_q \to {\mathbf{Y}}_q$ is a weak equivalence for all $q$. Then the induced map
\[
\diag_q f_q\colon \diag_q {\mathbf X}_q \to \diag_q {\mathbf Y}_q
\]
is a weak equivalence.
\end{cor}
Note that in the model structures that we consider on $\sps(C)$ the assumption on cofibrancy is always satisfied.

\begin{lemma}\label{diagmap}
Let ${\mathbf X}$ be a simplicial object and $A$ a discrete object in $\sps(C)$. Then there is a natural isomorphism
\[
\Map(A, \diag {\mathbf X}) \cong \diag \Map(A, {\mathbf X}).
\]
\end{lemma}

\begin{proof}
Using the fact that evaluation at level 0 is left adjoint to the discrete inclusion $\set\hookrightarrow \sset$, we get a chain of natural bijections giving us the desired isomorphism levelwise:
\[
\begin{split}
\Map(A, \diag_q {\mathbf X}_q)_p
& \cong \Hom_{\sps(C)}(A\times \Delta[p], \diag_q {\mathbf X}_q)\\
& \cong \Hom_{\sps(C)}(A , (\diag_q {\mathbf X}_q)^{\Delta[p]} )\\
& \cong \Hom_{\ps(C)}(A , ((\diag_q {\mathbf X}_q)^{\Delta[p]})_{\cdot,0} )\\
& \cong \Hom_{\ps(C)}(A , ((\diag_q {\mathbf X}_q)^{})_{\cdot,p} )\\
& \cong \Hom_{\ps(C)}(A ,  {(\mathbf X}_p)_{\cdot,p} )\\
& \cong \Hom_{\ps(C)}(A ,  (({\mathbf X}_p)^{\Delta[p]})_{\cdot,0} )\\
& \cong \Hom_{\sps(C)}(A ,  {(\mathbf X}_p)^{\Delta[p]} )\\
& \cong \Hom_{\sps(C)}(A \times \Delta[p] ,  {\mathbf X}_p )\\
& \cong \Map(A  ,  {\mathbf X}_p )_p\\
& \cong (\diag_q \Map(A  ,  {\mathbf X}_{q} ))_p.\\
\end{split}
\]
\end{proof}


Despite being a homotopy colimit, the diagonal also has some compatibility with limits and homotopy limits.

\begin{lemma}\label{diaglimits}
The simplicial diagonal commutes with limits.
\end{lemma}

\begin{proof}
The result follows by from the fact that $\diag$ is given by precomposition by the diagonal $\Delta\to \Delta\times \Delta$.
\end{proof}

\begin{lemma}\cite[8.20]{brcomp2}\label{pbpasting}
Let $\mathbf{X}\to \mathbf{Y}$ be a map of simplicial objects of $\sps(C)$ such that, for every map $[q]\to [p]$ in $\Delta$, the induced diagram 
\[
\begin{tikzcd}
{\mathbf{X}}_q \arrow[r] \arrow[d] & {\mathbf{X}}_p \arrow[d]\\
{\mathbf{Y}}_q \arrow[r] & {\mathbf{Y}}_p \\
\end{tikzcd}
\]
is a levelwise homotopy pullback square. Then the diagram
\[
\begin{tikzcd}
{\mathbf{X}}_0 \arrow[r] \arrow[d] & \displaystyle \hocolim_{q\in \Delta}{\mathbf{X}}_q \arrow[d]\\
{\mathbf{Y}}_0 \arrow[r] & \displaystyle  \hocolim_{q\in \Delta}{\mathbf{Y}}_q \\
\end{tikzcd}
\]
is also a levelwise homotopy pullback square.
\end{lemma}

We are now ready prove that the dimension 1 completion satisfies properties analogous to the Rezk completion \cite[\S 14]{rezksp}.

\begin{thm}\label{hcompletion}
Let $X$ be a $\se(S)$-fibrant object in $\sps(\Theta C)$. Then
\begin{enumerate}
\item $T^1 X$ is $\cpt(S)$-fibrant,
\item $\eta_X^1$ is a $\cpt(S)$-local acyclic cofibration, and
\item $\eta_X^1$ is a DK-equivalence.
\end{enumerate}
\end{thm}

\begin{proof} 
We start with the second property by decomposing $\eta_X^1$ into three maps, each of which is a $\cpt(S)$-local equivalence. All functors $I(q)\to I(p)$ induced by maps $[q]\to [p]$ in $\Delta$ are equivalences of categories and are thus taken to categorical equivalences $\pi^*E(q)\to \pi^*E(p)$ by the discrete nerve by Corollary \ref{eqcat}. Then by Corollary \ref{cateqcart} the maps $X^{\pi^*E(p)}\to X^{\pi^*E(q)}$ are also categorical equivalences between $\se(S)$-fibrant objects and hence $\cpt(S)$-local equivalences by Proposition \ref{cateqcpt}. In particular, when $p=0$ we get maps $X\to X^{\pi^*E(q)}$ whose homotopy colimit in the $\cpt(S)$-local model structure then remains a $\cpt(S)$-local equivalence 
\[
X=\hocolim_{[q]\in\Delta}X \overset{\sim}{\to} \hocolim_{[q]\in\Delta} X^{\pi^*E(q)},
\]
which is our first map in the decomposition.
The second is the Bousfield-Kan map, which applied in the $\cpt(S)$-local model category, is a $\cpt(S)$-local equivalence by Theorem \ref{bousfieldkan}:
\[
\hocolim_{[q]\in\Delta} X^{\pi^*E(q)}\overset{\sim}{\to} \diag_q X^{\pi^*E(q)}.
\]
Finally we have the injective fibrant replacement, which is a levelwise equivalence and hence also a $\cpt(S)$-local equivalence:
\[
\diag_q X^{\pi^*E(q)}=\widetilde{T}^1 X \overset{\sim}{\to} T^1 X.
\]
The map $\eta_X^1$ is the composite of the three maps above and thus a $\cpt(S)$-local equivalence.

Also note that for each $q$ the map $X\to X^{\pi^*E(q)}$ is a monomorphism, since it has a left inverse induced by any map $F[0]\to \pi^*E(q)$. Thus the map $X\to \diag_q X^{\pi^*E(q)}$ is also a monomorphism as it shares its levels with the preceding monomorphisms. The injective fibrant replacement is also a monomorphism, which means that so is $\eta_X^1$. Thus $\eta_X^1$ is a $\cpt(S)$-local acyclic cofibration.

Next we consider the behaviour of $\eta_X^1$ on mapping objects.
The maps $f\colon X^{\pi^*E(p)}\to X^{\pi^*E(q)}$ are categorical equivalences between $\se(S)$-fibrant objects, so they are also DK-equivalences by Proposition \ref{cateqdk}. Hence we have levelwise weak equivalences on the mapping objects 
\[
\begin{tikzcd}
f_* \colon \map_{X^{\pi^*E(p)}}(x,y) \arrow[r,"\sim"] & \map_{X^{\pi^*E(q)}}(f(x),f(y))
\end{tikzcd}
\]
and thus on their products as well.
By Cartesianness of the $\se(S)$-local model structure the objects $X^{\pi^*E(q)}$ are $\se(S)$-fibrant, so
$f$ also induces levelwise weak equivalences 
\[
\begin{tikzcd}
\map_{X^{\pi^*E(p)}}(x_0,\ldots,x_m) \arrow[r,"\sim"] & \map_{X^{\pi^*E(q)}}(f(x_0),\ldots,f(x_m))
\end{tikzcd}
\]
for $x_0,\ldots,x_m\in (X^{\pi^*E(p)} )_{[0]}$.
Then the square below is a homotopy pullback square for any $c_1,\ldots,c_m\in C$, since the levels of the above map are the ones induced on the homotopy fibers of the vertical maps in the diagram
\[
\begin{tikzcd}[sep=large]
(X^{\pi^*E(p)})_{[m](c_1,\ldots,c_m)} \arrow[r, "f_{[m](c_1,\ldots,c_m)}"] \arrow[d, two heads] \arrow[dr, phantom, "\lrcorner_h" very near start]
& (X^{\pi^*E(q)})_{[m](c_1,\ldots,c_m)} \arrow[d, two heads]\\
(X^{\pi^*E(p)})_{[0]}^{m+1} \arrow[r, "f_{[0]}^{m+1}"] & (X^{\pi^*E(q)})_{[0]}^{m+1}.
\end{tikzcd}
\]
Lemma \ref{pbpasting} tells us then that taking the simplicial diagonal results in the left square below also as a homotopy pullback:
\[
\begin{tikzcd}
X_{[m](c_1,\ldots,c_m)} \arrow[r] \arrow[d, two heads] \arrow[dr, phantom, "\lrcorner_h" very near start] \arrow[rr, bend left, "(\eta_X^1)_{[m](c_1,\ldots,c_m)}"]
& (\widetilde{T}^1 X)_{[m](c_1,\ldots,c_m)} \arrow[d] \arrow[r, "\sim"]
& (T^1 X)_{[m](c_1,\ldots,c_m)} \arrow[d, two heads]\\
X_{[0]}^{m+1} \arrow[r] \arrow[rr, bend right, "(\eta_X^1)_{[0]}^{m+1}" below]
& (\widetilde{T}^1 X)_{[0]}^{m+1} \arrow[r, "\sim"]
& (T^1 X)_{[0]}^{m+1}.
\end{tikzcd}
\]
Since the horizontal maps are levelwise weak equivalences in the right square, it is also a homotopy pullback as is the outer rectangle.
Looking at the homotopy fibers of the vertical maps, we get levelwise weak equivalences
\[
\begin{tikzcd}
(\eta_X^1)_* \colon \map_{X}(x_0,\ldots,x_m) \arrow[r,"\sim"] & \map_{T^1 X}(\eta_X^1(x_0),\ldots,\eta_X^1(x_m)),
\end{tikzcd}
\]
which for $m=1$ tells us that $\eta_X^1$ induces weak equivalences on mapping objects, or in other words, that $\eta_X^1$ is homotopically fully faithful.
For higher $m$ the left side is weakly equivalent to the iterated product of the mapping objects and thus so is the right side, which implies that $T^{1}X$ is $\se$-fibrant. Since the mapping objects are injective fibrant and levelwise weakly equivalent to the $S$-fibrant mapping spaces of $X$, the completion $T^{1}X$ is $\se(S)$-fibrant.

Next we show that $T^1 X$ is $\cpt$-local. We can restrict the weak equivalences on the mapping spaces to the components consisting of homotopy equivalences to obtain the following diagram:
\[
\begin{tikzcd}
X_{eq} \arrow[r] \arrow[d, two heads, "(d_1{,}d_0)" left] \arrow[dr, phantom, "\lrcorner_h" very near start] \arrow[rr, bend left, "(\eta_X^1)_{eq}"]
& (\widetilde{T}^1 X)_{eq} \arrow[d, "(d_1{,}d_0)"] \arrow[r, "\sim"]
& (T^1 X)_{eq} \arrow[d, two heads, "(d_1{,}d_0)"]\\
X_{[0]}\times X_{[0]} \arrow[r] \arrow[rr, bend right, "(\eta_X^1)_{[0]}\times (\eta_X^1)_{[0]}" below]
& (\widetilde{T}^1 X)_{[0]}\times (\widetilde{T}^1 X)_{[0]} \arrow[r, "\sim"]
& (T^1 X)_{[0]}\times (T^1 X)_{[0]} .
\end{tikzcd}
\]

Then using the discreteness of $\pi^*E$, we may commute the diagonal with mapping spaces by Lemma \ref{diagmap} to obtain the following chain of natural isomorphisms:
\[
\begin{split}
\Map(\pi^*E, \widetilde{T}^1 X )
& \cong \Map(\pi^*E, \diag_q X^{\pi^*E(q)} )\\
& \cong  \diag_q\Map(\pi^*E, X^{\pi^*E(q)} )\\
& \cong \diag_q (X^{\pi^*E(q)\times \pi^*E})_{[0]}\\
& \cong (\diag_q(X^{\pi^*E})^{\pi^*E(q) })_{[0]}\\
& = (\widetilde{T}^1 (X^{\pi^*E}))_{[0]}\\
& \cong \Map(F[0],\widetilde{T}^1 (X^{\pi^*E}) ).
\end{split}
\]
In particular, there is a weak equivalence $(\widetilde{T}^1 (X^{\pi^*E}))_{[0]} \overset{\sim}{\to}(\widetilde{T}^1 X)_{eq}\subseteq (\widetilde{T}^1 X)_{[1](t)}$, and thus also $(T^1(X^{\pi^*E}))_{[0]} \overset{\sim}{\to}(T^1 X)_{eq}$. Now since the map $X\to X^{\pi^*E(q)}$ induced by $\pi^*E\to F[0]$ is a categorical equivalence, it becomes a levelwise equivalence on the completions  and at level $[0]$ fits the following diagram on the left:
\[
\begin{tikzcd}
(T^1 X)_{[0]} \arrow[r, "\sim"] \arrow[rr, bend right, "s_0" below]
& (T^1(X^{\pi^*E}))_{[0]}  \arrow[r, "\sim"] 
& (T^1 X)_{eq},
\end{tikzcd}
\]
which tell us precisely that $T^1 X$ is local with respect to the map $\pi^*E\to F[0]$, that is, $\cpt$-local.
Since we have already shown $T^1 X$ to be $\se(S)$-fibrant, it is then $\cpt(S)$-fibrant.

It remains to show that $\eta_X^1$ is essentially surjective on homotopy categories. However, since we have shown the codomain to be $\cpt(S)$-fibrant, it suffices to show surjectivity on $\pi_0$.
Note that the precompletion acts as the identity on objects $X_{{[0]},0}=(\widetilde{T}^1 X)_{{[0]},0}$, and on higher simplicies of level ${[0]}$ we have an injections
\[
X_{{[0]},p}=\Hom(\Delta[p],X)\hookrightarrow \Hom(\pi^*E(p)\times \Delta[p],X)=(\widetilde{T}^1 X)_{{[0]},p}.
\]
Thus the induced map on the space of paths in level $[0]$ is an injection, which implies that the precompletion is surjective: $\pi_0(X_{[0]})\twoheadrightarrow \pi_0((\widetilde{T}^1 X)_{[0]})$. Additionally, the injective fibrant replacement is, in particular, a weak equivalence at level $[0]$, which implies that composing with it retains surjectivity on $\pi_0$ for $\eta_X^1$:
\[
\pi_0(X_{[0]})\twoheadrightarrow \pi_0((T^1 X)_{[0]}).
\]
Thus $\eta_X^1$ is a DK-equivalence, concluding the proof.

\end{proof}

We can now prove a restricted variant of Theorem \ref{t3}, where we assume fibrancy of the mapping objects. For $C=\Theta_{n-1}$ and $S=\cpt^{n-1}$, this theorem was proven by Bergner in \cite[6.4]{bdisc} via a comparison functor to a different model of $(\infty,n)$-categories based in $\sps(\Delta\times\Theta_{n-1})$. 
\begin{thm}\label{dkmain}
A map between $\se(S)$-fibrant objects of $\sps(\Theta C)$ is a DK-equivalence if and only if it is a $\cpt(S)$-local equivalence.
\end{thm}

\begin{proof}
Let $f\colon X\to Y$ be a map of $\se(S)$-fibrant objects of $\sps(\Theta C)$. Then we have the following commutative diagram:
\[
\begin{tikzcd}
X \arrow[r,"f" above] \arrow[d, "\eta_X^1" , "\sim" swap]   & Y \arrow[d, "\eta_Y^1", "\sim" swap]\\
T^1 X \arrow[r,"T^1 f" below] & T^1 Y.
\end{tikzcd}
\]
Now we have the following chain of equivalent statements:
\begin{enumerate}
\item $f$ is a $\cpt(S)$-local equivalence,
\item $T^1 f$ is a $\cpt(S)$-local equivalence,
\item $T^1 f$ is a levelwise equivalence,
\item $T^1 f$ is a DK-equivalence, and
\item $f$ is a DK-equivalence.
\end{enumerate}
Here the equivalences $(1)\Leftrightarrow (2)$ and $(4)\Leftrightarrow (5)$ are due to two-out-of-three properties;
$(2)\Leftrightarrow (3)$ follows from $\cpt(S)$-fibrancy of $T^1 X$ and $T^1 Y$ obtained by Theorem \ref{hcompletion}; and $(3)\Leftrightarrow (4)$ follows by Proposition \ref{dklw}.

\end{proof}

\section{Higher equivalences}\label{hdk}

In this section we consider higher-dimensional analogues of the homotopy and DK-equivalences from Section \ref{sdk}. Our definitions are inspired by the descriptions of equivalences of higher categories found in \cite{oreq}.


\begin{df}\label{hheq}
Let $X$ be a $\se^{n+1}(C)$-fibrant object $\sps(\Theta^{n+1} C)$ and $x,y\in X_{[0],0}$.
Let $\varphi\in \map_X(x,y)_{[0],0}$ be a 1-cell in $X$ and let $n\geq 0$. We say that $\varphi$ is a \emph{categorical $n$-equivalence} from $x$ to $y$ if there are $\psi,\psi'\in \map_X(y,x)_{[0],0}$ together with categorical $(n-1)$-equivalences from $\varphi \psi$ to $\id_y$ and $\psi'\varphi$ to $\id_x$ for some composites $\varphi \psi$ and $\psi'\varphi$, as illustrated in the diagram
\[
\begin{tikzcd}
& x 
\arrow[rr, "\psi'\varphi" description, ""{name=xm, above}] 
\arrow[rr, equal, bend left=60, ""{name=xid, below}, "s_0 x" above]
\arrow[Rightarrow,to path=(xm) -- (xid)\tikztonodes, "\simeq_{n-1}" swap]
\arrow[rdd, "\varphi" {description}] 
 & & x \\
& & &    \\
y 
\arrow[ruu, "\psi" description] 
\arrow[rr, "\varphi \psi" description, ""{name=ym, below}] 
\arrow[rr, equal, bend right=60, ""{name=yid, above}, "s_0 y" below]
\arrow[Rightarrow,to path=(ym) -- (yid)\tikztonodes, "\simeq_{n-1}"]
  &  & 
y.
\arrow[ruu, "\psi' " description] & \\ 
\end{tikzcd}
\]
We define categorical $-1$-equivalence as a path, so that a categorical $0$-equivalence is a homotopy equivalence.

We say that $x$ and $y$ are \emph{categorically $n$-equivalent}, denoted $x\simeq_n y$, if there is a zig-zag of categorical $n$-equivalences between them
\end{df}

In \cite[Section 10]{rezkth}, Rezk provides an alternate characterization of completeness using a corepresenting object for homotopy equivalences, which after passing through the inclusion ${\tau_!\colon\sps(\Delta)\hookrightarrow \sps(\Theta C)}$ is given by the pushout
\[
\begin{tikzcd}[sep=large]
F[1](t)\coprod F[1](t) 
\arrow[d, "{(\delta^{13},\delta^{02})}" left] 
\arrow[r, "\sigma^0\coprod \sigma^0" above] 
\arrow[dr, phantom, "\ulcorner" very near end]
& F[0]\coprod F[0] 
\arrow[d]\\
 F[3](t,t,t) 
 \arrow[r]
& \mathcal{E}^0,
\end{tikzcd}
\]
where maps $\delta$ are the face maps whose image does not include the vertices in the index.

If instead of collapsing the two edges corresponding to composite 1-cells to degenerate ones, we glue in categorical $n-1$-equivalences to the degenerate edges, we inductively obtain an object corepresenting categorical $n$-equivalences.

\begin{con}[Walking categorical $n$-equivalence]
Let $\mathcal{E}^n $ be the object of $\sps(\Theta^n C)$ defined as the colimit of the diagram
\begin{equation}\label{neqd}
\begin{tikzcd}[sep=huge]
&  F[1]([0])\coprod F[1]([0]) 
\arrow[r, "\sigma^0\coprod \sigma^0" above]
\arrow[d, "{V[1](j_1)\coprod V[1](j_1)}" left]
& F[0]\coprod F[0]\\
F[1]([0])\coprod F[1]([0]) 
\arrow[d, "{(\delta^{13},\delta^{02})}" left] 
\arrow[r, "{V[1](j_0)\coprod V[1](j_0)}" {above, yshift=5pt}] & 
V[1](\mathcal{E}^{n-1})\coprod V[1](\mathcal{E}^{n-1})  & \\
 F[3]([0],[0],[0]), &    & 
\end{tikzcd}
\end{equation}
where $j_0,j_1\colon F[0]\to \mathcal{E}^{n-1}$ are the start and endpoint of the 1-cell $i\colon F[1]([0])\hookrightarrow \mathcal{E}^{n-1}$ induced by the map $\delta^{03}\colon [1]\to [3]$. In the above diagram we have also used the identification $V[1](F[0])\cong F[1]([0])$. Observe that the vertical maps in (\ref{neqd}) are monomorphisms, so the colimit of the diagram is also a model for the homotopy colimit.

\end{con}
Now a 1-cell $\varphi\colon F[1]([0])\to X$ in a $\se^n(C)$-fibrant object $X$ is a categorical $n$-equivalence precisely if it admits an extension
\[
\begin{tikzcd}[sep=large]
F[1]([0])
\arrow[d, "j" left]
\arrow[r, "\varphi"] 
& X. \\
\mathcal{E}^n 
\arrow[ru, dashed, "\exists \hat{\varphi}" swap]
&
\end{tikzcd}
\]

Next we show homotopy invariance of categorical $n$-equivalences.
\begin{lemma}\label{neqh}
Let $\varphi$ be a categorical $n$-equivalence and $\varphi'$ homotopic to $\varphi$. Then $\varphi'$ is also a categorical $n$-equivalence.
\end{lemma}

\begin{proof}
The case $n=-1$ follows by definition, and the case $n=0$ is shown in \cite[Lemma 5.8]{rezksp} whose technique we apply to the general case. Let $X$ be $\se^{n+1}(C)$-fibrant, $\varphi,g\colon F[1]([0]) \to X$ with $\varphi$ a categorical $n$-equivalence and $H$ a homotopy between $\varphi$ and $\varphi'$. Since the map $i$ is a cofibration and $X$ is injective fibrant, precomposition by $i$ is a Kan fibration on the mapping spaces to $X$. Thus the homotopy $H$ has a lift $\hat{H}$ as indicated in the diagram
\[
\begin{tikzcd}[sep=large]
\Delta[0]
\arrow[d, hook, "j_0" left, "\sim" right] 
\arrow[r, "\hat{\varphi}" above] 
& \Map(\mathcal{E}^n,X) 
\arrow[d, two heads, "i^*"]\\
 \Delta[1] 
 \arrow[r, "H" below]
 \arrow[ru, dashed, "\exists \hat{H}"]
& \Map(F[1]([0]),X).
\end{tikzcd}
\]
Then $\hat{H} j_1$ is an extension of $\varphi'$ to $\mathcal{E}^n$, showing that it is also a categorical $n$-equivalence.
Also note that if $H$ is a homotopy relative to the endpoints, then so is $\hat{H}$.
\end{proof}
The above lemma in particular implies that the definition of a categorical $n$-equivalence is independent of the choice of composites. This fact can be used in a straightforward computation to show that categorical $n$-equivalences are closed under composition and when suspended also under whiskering.

Closure under homotopy also gives us the following result analogous to Lemma \ref{heqmono}.
\begin{cor}\label{hneq}
Every homotopy equivalence is a categorical $k$-equivalence for all $0\leq k\leq n$. Furthermore, if we denote the subspace of $X_1$ consisting of categorical $k$-equivalences in a $\se^{n+1}(C)$-fibrant $X$ by $X_{eq^k}$, then the degeneracy map $s_0\colon X_0\to X_1$ factors as the tower
\[
X_{[0]} \simeq X_{eq^{-1}}\hookrightarrow X_{eq^0}\hookrightarrow X_{eq^1} \hookrightarrow\cdots \hookrightarrow X_{eq^n}\hookrightarrow X_1
\]
of homotopy monomorphisms.
\end{cor}

Next, we provide an alternative, symmetric definition for categorical $n$-equivalences, which implies zig-zags may always be reduced to a single, direct equivalence.

\begin{prop}
A 1-cell $\varphi\in \map_X(x,y)_{[0],0}$ is a categorical $n$-equivalence in $X$ if and only if there is a 1-cell $\psi\in \map_X(y,x)_{[0],0}$ so that $\varphi \psi$ is categorically $(n-1)$-equivalent $\id_y$ and $\psi\varphi$ to $\id_x$.
\end{prop}

\begin{proof}
The case $n=0$ is Proposition \ref{eeq}, which we use as an inductive base case. Suppose then that the claim is true for $n-1$. The backwards implication of the claim for $n$ then follows immediately, setting $\psi'=\psi$ in the definition, since $\psi\varphi$ and $\varphi\psi$ being categorically $(n-1)$-equivalent to the respective identities implies the existence of a direct equivalences by the induction hypothesis. For the forwards implication, we observe that the two one-sided inverses are equivalent: $\psi'\simeq_{n-1} \psi' \varphi \psi \simeq_{n-1} \psi$ in $\map_X(y,x)$. Thus $\psi\varphi \simeq_{n-1} \id_x$, making $\psi$ an inverse to $\varphi$ also on the right.
\end{proof}

Next, we want to show that the tower 
\[
X_{[0]} \simeq X_{eq^{-1}}\hookrightarrow X_{eq^0}\hookrightarrow X_{eq^1} \hookrightarrow\cdots \hookrightarrow X_{eq^n}
\]
collapses for complete objects.

\begin{lemma}
The unique map $\mathcal{E}^n \to F[0]$ is a $\cpt^{n+1}(C)$-local equivalence.
\end{lemma}

\begin{proof}
We proceed by induction on $n$, where the initial case $n=0$ follows by \cite[10.1]{rezkth}. Suppose then that the map $r\colon \mathcal{E}^{n-1} \to F[0]$ is a $\cpt^{n}(C)$-local equivalence. Then its suspension is a $\se\cpt^{n}(C)$-local equivalence, giving us the equivalence of diagrams
\[
\begin{tikzcd}[column sep=6em]
F[0]\coprod F[0] \arrow[r,equal]
& F[0]\coprod F[0] \\
F[1]([0])\coprod F[1]([0]) \arrow[r,equal] \arrow[u] \arrow[d]
& F[1]([0])\coprod F[1]([0]) \arrow[u] \arrow[d, "\cong" right]\\
V[1](\mathcal{E}^{n-1})\coprod V[1](\mathcal{E}^{n-1}) \arrow[r, "{V[1](r)\coprod V[1](r)}", "\sim" swap]
& V[1](F[0])\coprod V[1](F[0]) \\
F[1]([0])\coprod F[1]([0]) \arrow[r,equal] \arrow[u] \arrow[d]
& F[1]([0])\coprod F[1]([0]) \arrow[u, "\cong" right] \arrow[d]\\
F[3]([0],[0],[0]) \arrow[r,equal]
& F[3]([0],[0],[0]), \\
\end{tikzcd}
\]
which, in turn, induces a $\se\cpt^{n}(C)$-local equivalence of the homotopy colimits 
$\begin{tikzcd}[sep=small]
\mathcal{E}^{n} \arrow[r, "\sim"] & \mathcal{E}^0.
\end{tikzcd}$
Now the composite map $\mathcal{E}^{n} \to \mathcal{E}^0 \to F[0]$ is a $\cpt^{n+1}(C)$-local equivalence as wanted.
\end{proof}

\begin{prop}\label{neqcpt}
Let $X$ be $\cpt^{n+1}(C)$-fibrant. Then $x$ and $y$ are categorically $n$-equivalent in $X$ if and only if they are in the same connected component of $X_{[0]}$.
\end{prop}

\begin{proof}
Let $\operatorname{Path}_X(x,y)$ and $\heq^n_X(x,y)$ denote the spaces of paths and categorical $n$-equivalences from $x$ to $y$, respectively. Now in analogy to \cite[Theorem 6.2]{rezksp}, we have the diagram
\[
\begin{tikzcd}
\operatorname{Path}_X(x,y) \arrow[r] \arrow[d]
& \Map(F[0],X)\arrow[d, "r^*"]\\
\heq^n_X(x,y) \arrow[r] \arrow[d]
& \Map(\mathcal{E}^{n},X)\arrow[d, "{(\alpha^*, \omega^*)}"]\\
\{*\} \arrow[r, "{(x,y)}"] & \Map(F[0]\coprod F[0],X),
\end{tikzcd}
\]
where the outer rectangle and the bottom square are homotopy pullbacks. Furthermore, if $X$ is $\cpt^{n+1}(C)$-fibrant, then $r^*$ is a weak equivalence and thus induces a weak equivalence of the homotopy limits. In particular, $\operatorname{Path}_X(x,y)$ is non-empty if and only if $\heq^n_X(x,y)$ is non-empty.

\end{proof}

We are now ready to discuss the higher-dimensional DK-equivalences.

\begin{df}\label{dkn}
We say that a morphism $f\colon X\to Y$ between $\se^n(C)$-fibrant objects of $\sps(\Theta^n C)$ is a \emph{DK-$n$-equivalence} if the map 
\[
f^*\colon \map_X^k(\alpha,\beta)\to \map_Y^k(f(\alpha),f(\beta))
\]
is surjective up to a categorical $(n-k-1)$-equivalence on objects for $0\leq k\leq n-1$ and a levelwise weak equivalence for $k=n$, for all $(k-1)$-cells $\alpha,\beta$ in $X$.
\end{df}

The above definition may be rephrased inductively as follows.
\begin{lemma}
A morphism $f\colon X\to Y$ is a DK-$n$-equivalence if and only if
\begin{enumerate}
\item it is surjective up to categorical $(n-1)$-equivalence on objects, and 
\item the map 
\[
f^*\colon \map_X(x,y)\to \map_Y(f(x),f(y))
\]
is a DK-$(n-1)$-equivalence for all $x,y\in X_{[0],0}$.
\end{enumerate}
\end{lemma}
Observe that as we increase $n$, the notions of essential surjectivity and fully faithfulness have to be both weakened. This definition gives the correct equivalences between $(\infty,n)$-categories, as indicated by the following proposition.


\begin{prop}
Let $X$ and $Y$ be $\cpt^n(C)$-fibrant. Then $f\colon X\to Y$ is a levelwise weak equivalence if and only if it is a DK-$n$-equivalence.
\end{prop}

\begin{proof}
We proceed by induction. The case $n=1$ is shown in Proposition \ref{dklw}. Suppose then that the claim is true for some $k=n-1$. Now $f$ is a DK-$n$-equivalence if and only if 
\begin{enumerate}
\item it is surjective on $\pi_0$ on objects, and 
\item the map 
\[
f^*\colon \map_X(x,y)\to \map_Y(f(x),f(y))
\]
is a levelwise weak equivalence for all $x,y\in X_{[0],0}$,
\end{enumerate}
where the condition $(1)$ is equivalent to surjectivity up to categorical $(n-1)$-equivalence by Proposition \ref{neqcpt}. However, these two conditions precisely define a DK-1-equivalence, so the claim follows by the initial case.
\end{proof}

In order to utilize the proof technique from Theorem \ref{dkmain} for our main theorem \ref{t3}, we need to show the two-out-of-three property for higher DK-equivalences. For this, we require the following lemma. 

\begin{lemma}\label{eqref}
Let $f\colon X\to Y$ be a DK-$n$-equivalence. Then $x$ and $y$ are categorically $(n-1)$-equivalent if and only if $f(x)$ and $f(y)$ are.
\end{lemma}

\begin{proof}
Since the data of categorical $(n-1)$-equivalences is preserved by any morphism, it suffices to show the backwards implication. The case $n=0$ is immediate and serves as an inductive base case. Suppose then that the claim is true for $n-1$ and let $\varphi\in \map_Y(f(x),f(y))_{[0],0}$ be a categorical $(n-1)$-equivalence in $Y$ with a two-sided inverse $\psi$. Now since $f$ is a DK-$n$-equivalence, it is surjective on 1-cells up to categorical $(n-2)$-equivalence. Thus there are 1-cells $\hat{\varphi}\in \map_X(x,y)_{[0],0}$ and $\hat{\psi}\in \map_X(y,x)_{[0],0}$ with categorical $(n-2)$-equivalences $f_*(\hat{\varphi})\simeq_{n-2} \varphi$ and $f_*(\hat{\psi})\simeq_{n-2} \psi$. Consequently, we also have that 
\[
 f_*(\hat{\psi}\hat{\varphi})\simeq_{n-2}\psi \varphi \simeq_{n-2} \id_x \  \text{and}\  f_*(\hat{\varphi}\hat{\psi})\simeq_{n-2}\varphi \psi \simeq_{n-2} \id_y,
\]
which by induction hypothesis implies that $\hat{\psi}\hat{\varphi} \simeq_{n-2} \id_x$ and $\hat{\varphi}\hat{\psi} \simeq_{n-2} \id_y$. Thus $\hat{\varphi}$ is a categorical $(n-1)$-equivalence between $x$ and $y$.
\end{proof}

\begin{prop}\label{dkn23}
DK-$n$-equivalences satisfy the two-out-of-three property.
\end{prop}

\begin{proof}
The case $n=0$ is immediate, so by induction it suffices to consider two-out-of-three for the essential surjectivity part of the definition. Furthermore, of the three cases to consider two are straightforward, so we focus our discussion on the remaining one.
To this end, let $f\colon X\to Y$ and $g\colon Y\to Z$ be morphisms of $\se^n(C)$-fibrant objects, suppose that $g$ and $gf$ are DK-$n$-equivalences, and let $y$ be an object of $Y$.
By assumption $gf$ is essentially surjective, so there is an object $x$ in $X$ with an equivalence $gf(x)\simeq_{n-1} g(y)$. However, Lemma \ref{eqref} above implies then that $f(x)\simeq_{n-1} y$, since we assumed that $g$ is a DK-$n$-equivalence. Thus $f$ is also surjective up to categorical $(n-1)$=equivalence on objects, concluding the proof for this case.
\end{proof}

For the inductive proofs in the following section, we want to be able to compare higher DK-equivalences of varying dimension.

\begin{prop}
If $f$ is a DK-$m$-equivalence for $m < n$, then $f$ is a DK-$n$-equivalence.
\end{prop}

\begin{proof}
We proceed by induction on $n$, where it suffices to consider $m=n-1$ by a straightforward induction on $m$. The initial case $n=1$ follows by Lemma \ref{lwdk}. Suppose then that the claim is true for all $k\leq n-1$, and let $f$ be a DK-$k$-equivalence. Now $f$ is surjective up to categorical $(n-2)$-equivalence on objects, and thus also up to categorical $(n-1)$-equivalence by Corollary \ref{hneq}. Furthermore, on the mapping objects, $f$ induces DK-$(n-2)$-equivalences, which are also DK-$(n-1)$-equivalences by the induction hypothesis. Hence $f$ is a DK-$n$-equivalence as claimed.
\end{proof}

\section{Higher completions}\label{sscomp}

In this section, we inductively extend the completion functor to the higher dimensional completeness conditions, and prove that the higher dimensional completion functors satisfy Theorem \ref{t1}. We then use the completion to obtain a characterization of each of the completeness conditions in terms of DK-$n$-equivalences, analogously to Theorem \ref{dkmain} for $n=1$. The completions for different dimensions are sufficiently compatible with each other so that we may apply the completions successively to complete in all dimension and retain properties outlined in Theorem \ref{t2}.
Lastly, we use the combined completion to characterize completeness in terms of DK-$n$-equivalences as stated in Theorem \ref{t3}. All results in this section generalize those in Section \ref{shcomp}, and thus we can restrict our considerations to $n\geq 2$.

\begin{con}\label{ncompletion}
Consider the diagram $Q^1\colon \Theta C \times \Delta \to \sps(\Theta C)$ given by the formula
\[
Q_{\theta,p}^1= F(\theta) \times \pi^* E(p),
\]
which may be thought of as a chain of $p$ equivalences of diagrams of the shape $\theta$. Recall also that $\widetilde{T}^1 X_{\theta} \cong \diag_q \Map(Q_{\theta,q}^1 ,X)$. 
Suppose then that $X$ is an object of $\sps(\Theta^n C)$. In order to elevate the completion to the mapping objects of $X$, we want to apply the left adjoint of the mapping object functor, that is, the suspension $V[1]$ to the diagram $Q^1$. However, in order to make the completion uniform on all mapping objects in a way that is also compatible with the Segal conditions and to preserve objects, we instead have to use the more general intertwining functor $V$ for the following recursive definition:
\[
Q_{[m](\theta_1,\ldots,\theta_m),p}^n := V[m](Q_{\theta_1,p}^{n-1},\ldots,Q_{\theta_m,p}^{n-1}),
\]
where now $Q^n$ is a diagram $\Theta^n C \times \Delta \to \sps(\Theta^n C)$.

We can then define the functorial simplicial object $\mathbf{T}^n X$ with levelwise formula
\[
\mathbf{T}^n_q X_{\theta}=\Map(Q_{\theta,q}^n,X),
\]
and similarly to the $n=1$ case, we define the \emph{dimension $n$ precompletion}, denoted by $\widetilde{T}^n$, as the simplicial diagonal of $\mathbf{T}^n$. More explicitly, we have the formula
\[
\widetilde{T}^n X_{\theta,p}:=\mathbf{T}_p^n X_{\theta,p}=\Hom_{\sps(\Theta^n C)}(Q_{\theta,p}^n\times \Delta[p],X).
\]
We define the \emph{dimension $n$ completion} $T^n$ as the composite $\mathcal{F} \widetilde{T}^n$ for a fixed injective fibrant replacement functor $\mathcal{F}$.

Furthermore, the projections $F(\theta) \times \pi^* E(p)\to F(\theta)$ induce maps $Q_{\theta,p}^n\to F(\theta)$ and via precompositions natural maps $X\to \mathbf{T}_q^n X$, whose diagonal we denote by $\widetilde{\eta}_X^n\colon \id \to \widetilde{T}^n$. 
Then by postcomposing with the fibrant replacement maps we obtain a natural transformation $\eta^n\colon \id\to T^n$.
\end{con}
It may be verified that $T^n$ is a simplicial functor and $\eta^n$ a simplicial natural transformation.

Note that in addition to taking the simplicial diagonal of the objects $\mathbf{T}^n X$, the diagonal of $\Delta$ into its $m$-fold product also implicitly appears in recursion formula and breaks injective fibrancy in essentially all non-trivial cases.

We now state the precise formulation of Theorem \ref{t1}, which generalizes the properties of the dimension $1$ completion from Theorem \ref{hcompletion} to the general case. 

\begin{thm}\label{suspended}
Let $X$ be a $\se^n(S)$-fibrant object in $\sps(\Theta^n C)$. Then
\begin{enumerate}
\item $T^n X$ is $\se^{n-1}\cpt(S)$-fibrant,
\item $\eta_X^n$ is a $\se^{n-1}\cpt(S)$-local acyclic cofibration,
\item $\eta_X^n$ is a DK-$n$-equivalence, and
\item $(\tau_{n-2}^{n})^*\eta_X^n$ is a levelwise weak equivalence of the underlying $\Theta_{n-2}$-spaces.
\end{enumerate}
\end{thm}

Note that as an addition to the case $n=1$, the higher dimensional completions also keep the lower dimensional structure essentially invariant. In order to better organize the proof of this theorem, we separate each item into its own result. We start by proving property (4) as Corollary \ref{underlying} as it does not rely on the other properties. The key to the remaining properties is Proposition \ref{induction}, where we show that on the mapping objects, the dimension $n$ completion essentially acts as the dimension $n-1$ completion, which allows us to use induction with the properties from $n=1$ as the base case. The property (1) then follows as Corollary \ref{compfibrant}.

We prove the property (2) as Proposition \ref{compwe} in an indirect manner using two-out-of-three and two-out-of-six properties. A key to this approach is considering the second iteration of the completion and using the fact that the first iterate is already complete to find weak equivalences in the naturality diagrams for the completion map.
The last property (3) then follows as Corollary \ref{ndk} using the connection established between DK-1-equivalences and weak equivalences in Theorem \ref{dkmain} together with the recursive structure of the completions.

We prove property (4) by showing that the dimension $n$ precompletion preserves the underlying $\Theta_{n-2}$-space.

\begin{lemma}\label{uth}
Let $X$ be a $\se^n(S)$-fibrant object in $\sps(\Theta^n C)$. Then 
\[
(\tau_k^{n})^*\mathbf{T}_q^n X\cong  (\tau_k^{n})^* X
\]
for $k\leq n-2$ and any $[q]\in \Delta$.
\end{lemma}
 
\begin{proof}
We proceed inductively; first note that $Q_{[0],q}^l = V[0]\cong F[0]$ for any $l\geq 2$. Suppose then that, for some $k\geq 0$, $Q_{\tau_k^{l} \theta,q}^l\cong F(\tau_k^{l} \theta)$ for all $\theta\in \Theta_k$ . Now we have natural isomorphisms
\[
\begin{split}
Q_{\tau_{k+1}^{l+1} [m](\theta_1,\ldots,\theta_m),q}^{l+1}
& = Q_{ [m](\tau_{k}^{l+1}\theta_1,\ldots,\tau_{k}^{l+1}\theta_m),q}^n\\
& = V[m](Q_{\tau_{k}^{l+1}\theta_1,q}^{l},\ldots,Q_{\tau_{k}^{l+1}\theta_m,q}^{l}) \\
& \cong V[m](F(\tau_k^{l+1} \theta_1),\ldots,F(\tau_k^{l+1} \theta_m)) \\
& \cong F([m](\tau_k^{l+1} \theta_1,\ldots,\tau_k^{l+1} \theta_m)) \\
& \cong F(\tau_{k+1}^{l+1}[m]( \theta_1,\ldots, \theta_m)).
\end{split}
\]
Thus by induction $Q_{\tau_{k}^{n} \theta,q}^{n}\cong F(\tau_{k}^{n} \theta)$ for any $k\leq n-2$ and any $\theta\in \Theta_k$. Then we obtain the following natural isomorphisms
\[
\begin{split}
((\tau_k^{n})^*\mathbf{T}_q^n X)_{\theta}
& = (\mathbf{T}_q^n X)_{\tau_{k}^{n} \theta}\\
& \cong \Map(Q_{\tau_{k}^{n} \theta,q}^n,X)\\
& \cong \Map(F(\tau_{k}^{n} \theta),X)\\
& \cong X_{\tau_{k}^{n} \theta}\\
& = ((\tau_k^{n})^* X)_{\theta},
\end{split}
\]
which prove the claim.
\end{proof} 
 
Taking the simplicial diagonal of the isomorphisms in the above lemma then yields the property (4) of \ref{suspended}.

\begin{cor}\label{underlying}
Let $X$ be a $\se^n(S)$-fibrant object in $\sps(\Theta^n C)$. Then the precompletion induces an isomorphism
\[
(\tau_k^{n})^*\widetilde{T}^n X\cong  (\tau_k^{n})^* X,
\]
and the completion induces a levelwise weak equivalence
\[
(\tau_k^{n})^*T^n X\simeq (\tau_k^{n})^* X
\]
for $k\leq n-2$.
\end{cor}
Note that in particular when $k=0$, we have that 
\[
(\widetilde{T}^n X)_{[0]}\cong(\mathbf{T}_q^n  X)_{[0]}\cong X_{[0]} \simeq (T^n X)_{[0]}.
\]

Note we also have the natural isomorphism $(\widetilde{T}^n X)_{\theta,0}\cong X_{\theta,0}$, which we can use to obtain the following example illustrating the necessity of the injective fibrant replacement for the higher-dimensional completions even for discrete objects.
\begin{ex}
Consider the $3$-category $\Sigma^2 I$, which is gaunt in dimensions 1 and 2; that is, it has no non-identity isomorphisms in those dimensions, but it is not gaunt in dimension 3. Thus $N\Sigma^2 I\cong V[1]^2 E$ is $\cpt^2\se$-fibrant in $\sps(\Theta_3)$. Then supposing Theorem \ref{suspended}, $T^3 N\Sigma^2 I$ is $\cpt(\se\cpt)$-fibrant, which $\widetilde{T}^3 N\Sigma^2 I$ cannot be, since $(\widetilde{T}^3 N\Sigma^2 I)_{\cdot,0}\cong (N\Sigma^2 I)_{\cdot,0}$ is not fibrant in the model structure on $\ps(\Theta_3)$ with completeness in dimension 3, and since the evaluation at simplicial level 0 is right Quillen by \cite[8.4.]{ara}.
\end{ex}

\begin{lemma}\label{compsegal}
Let $X$ be a $\se$-fibrant object in $\sps(\Theta^n C)$. Then $T^n X$ is $\se$-fibrant.
\end{lemma}

\begin{proof}
First consider the following expression for the levels of $\mathbf{T}_q^{n} X$:
\[
\begin{split}
\mathbf{T}_q^{n} X_{[m](\theta_1,\ldots,\theta_m)}
& = \Map(Q^n_{[m](\theta_1,\ldots,\theta_m),q},X)\\
& = \Map(V[m](Q^n_{\theta_1,q},\ldots,Q^n_{\theta_m,q}),X).\\
\end{split}
\]
Since $X$ is $\se$-fibrant, we may decompose the intertwining functor into suspensions by using the generalized Segal maps from Proposition \ref{moresegal} to obtain the middle weak equivalence below:

\[
\begin{tikzcd}
\mathbf{T}_q^{n} X_{[m](\theta_1,\ldots,\theta_m)} \arrow[d, "\cong"]\\
\Map(V[m](Q^n_{\theta_1,q},\ldots,Q^n_{\theta_m,q}),X) \arrow[d, "\sim"] \\
\displaystyle \Map(V[1](Q^n_{\theta_1,q}),X)\tim_{X_{[0]}} \cdots \tim_{X_{[0]}} \Map(V[1](Q^n_{\theta_m,q}),X) \arrow[d, "\cong"] \\
\displaystyle \mathbf{T}_q^{n} X_{[1](\theta_1)}\tim_{\mathbf{T}_q^{n} X(q)_{[0]}} \cdots \tim_{\mathbf{T}_q^{n} X_{[0]}} \mathbf{T}_q^{n} X_{[1](\theta_m)},
\end{tikzcd}
\]
where we also use the fact that at level $[0]$ we have isomorphisms by Corollary \ref{underlying}.
Now since the diagonal preserves finite limits, the diagonal of the composite of the maps above is the corresponding Segal map for $\widetilde{T}^n X$:
\[
\begin{tikzcd}
\widetilde{T}^n X_{[m](\theta_1,\ldots,\theta_m)} \arrow[r] &
\displaystyle \widetilde{T}^n X_{[1](\theta_1)}\tim_{\widetilde{T}^n X_{[0]}} \cdots \tim_{\widetilde{T}^n X_{[0]}} \widetilde{T}^n X_{[1](\theta_m)},
\end{tikzcd}
\]
which is also a levelwise weak equivalence as a diagonal of levelwise weak equivalences by Corollary \ref{bousfieldkanwe}.

Then the Segal condition being preserved by levelwise weak equivalences implies that $T^n X$ is $\se$-fibrant.
\end{proof}

The next proposition establishes the recursive relation between the completion functors.

\begin{prop}\label{induction}
Let $X$ be a $\se$-fibrant object in $\sps(\Theta^n C)$. Then the map $\eta_X^n\colon X\to T^n X$ induces on the mapping objects is the component of $\eta^{n-1}$ up to a homotopy equivalence:
\[
\begin{tikzcd}[sep=large]
\map_{X}(x,y) \arrow[r, "{\eta_{\map_{X}(x,y)}^{n-1}}"] \arrow[dr, "{(\eta_{X}^{n})_*}" swap]
& T^{n-1}\map_{X}(x,y) \arrow[d, "\simeq"] \\
& \map_{T^n X}((\eta_X^n)_{[0]}(x),(\eta_X^n)_{[0]}(y)).
\end{tikzcd}
\]
\end{prop}

\begin{proof}
By the definition of the objects $\mathbf{T}_q^{n} X$ we have isomorphisms
\[
\begin{split}
\Map(V[1](F(\theta)),\mathbf{T}_q^{n} X)
& \cong \Map(Q^n_{[1](\theta),q}, X)\\
& \cong \Map(V[1](Q^{n-1}_{\theta,q}), X),
\end{split}
\]
and by Corollary \ref{underlying} pairs of objects agree:
\[
\Map(V[1](\emptyset),\mathbf{T}_q^{n} X) \cong \Map(V[1](\emptyset), X).
\]
The levels mapping objects of $\mathbf{T}_q^{n} X $ can then be computed as the fibers of the morphism
\begin{equation}\label{mapfiber}
 \Map(V[1](Q^{n-1}_{\theta,q}), X)\to  \Map(V[1](\emptyset),X),
\end{equation}
which by the suspension-map adjunction in Lemma \ref{mapgen} have the form
\[
\map_{\mathbf{T}_q^{n} X}(x,y)_{\theta}
 \cong \Map(Q^{n-1}_{\theta,q},\map_{ X}(x,y)).
\]
Taking the simplicial diagonal of the morphisms $\map_{X}(x,y)\to \map_{\mathbf{T}_q^{n} X}(x,y)$ then yields the dimension $n-1$ precompletion:
\[
\map_{X}(x,y)\to\diag_q \map_{\mathbf{T}_q^{n} X}(x,y) \cong \widetilde{T}^{n-1}(\map_{X}(x,y)).
\]

Next, we consider the mapping objects of $\widetilde{T}^n X$. Since $\widetilde{T}^n X$ is not necessarily injective fibrant, the mapping objects may a priori differ from the homotopy mapping objects, whereas only the latter are generally preserved by levelwise equivalences such as the injective fibrant replacement.

For objects, Corollary \ref{underlying} tells us that
\[
\Map(V[1](\emptyset),\widetilde{T}^n X) \cong \Map(V[1](\emptyset),  X),
\]
which is a Kan complex, so by using the formula for homotopy pullbacks in Lemma \ref{hmapd}, we may compute the homotopy mapping objects of $\widetilde{T}^n X$ as the limits of the diagrams
\[
\begin{tikzcd}
& &  \Map(V[1](F(\theta)),\widetilde{T}^n X) \arrow[d]\\
& \Map(V[1](\emptyset),X)^{\Delta[1]} \arrow[r, "j_0^*" above] \arrow[d, "j_1^*" left] 
& \Map(V[1](\emptyset),X)\\
\{*\} \arrow[r, "{(x,y)}" ] & \Map(V[1](\emptyset),X). &
\end{tikzcd}
\]
We may also use the fact that the diagonal is computed levelwise to obtain the isomorphism
\[
\Map(V[1](F(\theta)),\widetilde{T}^n X)\cong \diag_{q} \Map(V[1](F(\theta)),\mathbf{T}_q^{n} X).
\]

Now $\hmap_{\widetilde{T}^n X}(x,y)$ may be computed as the limit of the diagram of the diagonals
\[
\diag_{q}\left(\begin{tikzcd}
& &  \Map(V[1](F(\theta)),\mathbf{T}_q^{n} X) \arrow[d]\\
& \Map(V[1](\emptyset),X)^{\Delta[1]} \arrow[r, "j_0^*" above] \arrow[d, "j_1^*" left] 
& \Map(V[1](\emptyset),X)\\
\{*\} \arrow[r, "{(x,y)}" ] & \Map(V[1](\emptyset),X) &
\end{tikzcd}\right)
\]
and similarly $\map_{\widetilde{T}^n X}(x,y)$ as the limit of 
\[
\diag_{q}\left(\begin{tikzcd}
& \Map(V[1](F(\theta)),\mathbf{T}_q^{n} X) \arrow[d]\\
\{*\} \arrow[r, "{(x,y)}" ] & \Map(V[1](\emptyset),X) 
\end{tikzcd}\right).
\]
Using the commutativity of the diagonal with finite limits we then see that the canonical map $\map_{\widetilde{T}^n X}(x,y) \to \hmap_{\widetilde{T}^n X}(x,y)$ is the diagonal of the corresponding maps 
\begin{equation}\label{simphmap}
\map_{\mathbf{T}_q^{n} X}(x,y) \to \hmap_{\mathbf{T}_q^{n} X}(x,y).
\end{equation}
However, the (homotopy) mapping objects of $\mathbf{T}_q^{n} X$ can be computed in terms of the (homotopy) fibers of the map (\ref{mapfiber}), which is a fibration by injective fibrancy of $X$, telling us that fibers are weakly equivalent to the homotopy fibers. Then the diagonal of the weak equivalences (\ref{simphmap}) is also a levelwise weak equivalence by Corollary \ref{bousfieldkanwe}.

Now the fibrant replacement of $\widetilde{T}^n X$ induces a levelwise weak equivalence on the mapping objects by the two-out-of-three property:
\[
\begin{tikzcd}
\diag_q (\map_{\mathbf{T}_q^{n} X}(x,y)) \arrow[r, "\cong"] \arrow[d, "\sim"]
& \map_{\widetilde{T}^n X}(x,y) \arrow[r] \arrow[d]
&  \map_{T^n X}(x,y) \arrow[d, "\sim"] \\
\diag_q (\hmap_{\mathbf{T}_q^{n} X}(x,y)) \arrow[r, "\cong"]
& \hmap_{c}(x,y) \arrow[r, "\sim"]
&  \hmap_{T^n X}(x,y) .\\
\end{tikzcd}
\]

In Lemma \ref{compsegal} we showed that $T^n X$ is $\se$-fibrant, so its mapping objects are injective fibrant by Proposition \ref{mapfibrant}. Thus the map induced by $\eta_X^n$ on the mapping objects is isomorphic to
\[
\begin{tikzcd}[sep=large]
\map_{X}(x,y) \arrow[r, "{\widetilde{\eta}_{\map_X(x,y)}^n}"]
& \widetilde{T}^{n-1}(\map_{X}(x,y))
\end{tikzcd}
\]
followed by an injective fibrant replacement. The claim follows by the up-to-homotopy uniqueness of fibrant replacements.
\end{proof}

We then obtain property (1) of Theorem \ref{suspended}.

\begin{cor}\label{compfibrant}
Let $X$ be a $\se^n(S)$-fibrant object in $\sps(\Theta^n C)$. Then $T^n X$ is $\se^{n-1}\cpt(S)$-fibrant.
\end{cor}

\begin{proof}
We proceed by induction. The case $n=1$ is the item (1) of Theorem \ref{hcompletion}. Suppose then that the claim is true for some $k\geq 1$ and let $X$ be $\se^{k+1}(S)$-fibrant. Now $T^{k+1} X$ is $\se$-fibrant by Lemma \ref{compsegal} with $\se^{k}\cpt(S)$-fibrant mapping objects. Then $T^{k+1} X$ is $\se^{k}\cpt(S)$-fibrant by Proposition \ref{mapfibrant}, completing the induction.
\end{proof}

Next we show property (2) of Theorem \ref{suspended}.

\begin{prop}\label{compwe}
Let $X$ be a $\se^n(S)$-fibrant object in $\sps(\Theta^n C)$. Then $\eta_X^n$ is a $\se^{n-1}\cpt(S)$-local acyclic cofibration.
\end{prop}

\begin{proof}
First consider cofibrancy. The maps $X\to \mathbf{T}_q^{n} X$ are monomorphisms since they are levelwise defined as precomposition by a retraction. Then the diagonal $\widetilde{\eta}_X^n\colon X\to \widetilde{T}^n X$ is also a monomorphism, since the diagonal preserves finite limits. Then postcomposing with the injective fibrant replacement, which is also a monomorphism, gives the map $\eta_X^n\colon X\to T^n X$.

To show that $\eta_X^n$ is a $\se^{n-1}\cpt(S)$-local equivalence, we use induction on $n$ with the base case $n=1$ being part (2) of Theorem \ref{hcompletion}. Suppose then that the claim is true for some $k\geq 1$ and consider first $X$ a $\se^{k}\cpt(S)$-fibrant object of $\sps(\Theta^{k+1} C)$.
By Corollary \ref{underlying} $\eta_X^{k+1}$ is a weak equivalence on level $[0]$, and by Proposition \ref{induction} together with the induction hypothesis, $\eta_X^{k+1}$ induces a $\se^{k-1}\cpt(S)$-local acyclic cofibration on the mapping objects which are $\se^{k-1}\cpt(S)$-fibrant by assumption for $X$ and by Lemma \ref{compfibrant} for $T^{k+1}X$. Thus the map on mapping objects is a levelwise weak equivalence implying that we have homotopy pullback squares
\[
\begin{tikzcd}[sep=large]
X_{[1](\theta)} \arrow[r, "{(\eta_X^{k+1})_{[1](\theta)}}"] \arrow[d, two heads,  "{(d_1,d_0)}" swap] \arrow[dr, phantom, "\lrcorner_h" very near start]
& T^{k+1}X_{[1](\theta)} \arrow[d, two heads,  "{(d_1,d_0)}"] \\
X_{[0]}\times X_{[0]} \arrow[r, "{(\eta_X^{k+1})_{[0]}^2}", "\sim" swap] 
& T^{k+1}X_{[0]} \times T^{k+1}X_{[0]}
\end{tikzcd}
\]
for all $\theta\in \Theta^k C$. However, since the bottom map is a weak equivalence, so is the top map $(\eta_X^{k+1})_{[1](\theta)}$. It follows by the Segal condition for $X$ and $T^{k+1}X$ that the remaining levels of $\eta_X^{k+1}$ are also weak equivalences, so $\eta_X^{k+1}$ is $\se^{k-1}\cpt(S)$-local as a levelwise weak equivalence, concluding the case $n=k+1$ for $\se^{k}\cpt(S)$-fibrant $X$.

Consider then a $\se^{k+1}(S)$-fibrant $X$. By definition $\eta_X^{k+1}$ is a $\se^{k-1}\cpt(S)$-local equivalence if the morphism
\[
(\eta_X^{k+1})^* \colon \Map(T^{k+1} X, Z) \to \Map(X, Z)
\]
is a weak equivalence for all $\se^{k}\cpt(S)$-fibrant $Z$, which we employ the two-out-of-six property to show. 
Now we have the commutative diagram
\begin{equation}\label{wediagram}
\begin{tikzcd}[sep=large]
\Map(T^{k+1} X, Z) \arrow[r, "{T^{k+1}}"] \arrow[d, "(\eta_X^{k+1})^*" swap]
&  \Map(T^{k+1} T^{k+1} X, T^{k+1} Z) \arrow[d, "(T^{k+1}\eta_X^{k+1})^*"]\\
\Map(X, Z) \arrow[r, "{T^{k+1}}"] \arrow[rd, "(\eta_Z^{k+1})_*" swap]
&  \Map(T^{k+1} X, T^{k+1} Z) \arrow[d, "(\eta_X^{k+1})^*" ]\\
&  \Map( X, T^{k+1} Z),
\end{tikzcd}
\end{equation}
where the commutativity of the top square follows from the functoriality of $T^{k+1}$ and the bottom triangle by naturality of $\eta^{k+1}\colon \id\Rightarrow T^{k+1}$. Note that $(\eta_Z^{k+1})_*$ is a weak equivalence since $\eta_Z^{k+1}$ is a levelwise weak equivalence by the inductive step for $\se^{k}\cpt(S)$-fibrant objects shown above.

Similarly to the bottom triangle of the diagram (\ref{wediagram}), we also have
\[
\begin{tikzcd}
\Map(T^{k+1} X, Z) \arrow[r, "{T^{k+1}}"] \arrow[rd, "(\eta_Z^{k+1})_*" swap]
&  \Map(T^{k+1} T^{k+1} X, T^{k+1} Z) \arrow[d, "(\eta_{T^{k+1} X}^{k+1})^*" ]\\
&  \Map(T^{k+1} X, T^{k+1} Z),
\end{tikzcd}
\]
where both $(\eta_Z^{k+1})_*$ and $(\eta_{T^{k+1} X}^{k+1})^*$ are weak equivalences, since $Z$ and $T^{k+1} X$ are $\se^{k}\cpt(S)$-fibrant, using Corollary \ref{compfibrant}. Thus $T^{k+1}$ is a weak equivalence on the mapping space $\Map(T^{k+1} X, Z)$ by the two-out-of-three property.

Next consider the map $T^{k+1}(\eta_X^{k+1})$, which fits into the following commutative diagram by naturality of $\eta^{k+1}$:
\[
\begin{tikzcd}[sep=large]
X \arrow[r, "{\eta_X^{k+1}}"] \arrow[d, "{\eta_X^{k+1}}" swap]
& T^{k+1} X \arrow[d, "{\eta_{T^{k+1} X}^{k+1}}"] \\
T^{k+1} X \arrow[r, "{T^{k+1}\eta_X^{k+1}}"]
& T^{k+1} T^{k+1} X,
\end{tikzcd}
\]
where each component of $\eta^{k+1}$ is a weak equivalence on level $0$ and so is then $T^{k+1}\eta_X^{k+1}$ too by the two-out-of-three property. Furthermore, on the mapping objects the induced diagram is equivalent to the square in the following diagram by Proposition \ref{induction}:
\[
\begin{tikzcd}[sep=large]
& T^{k} \map_{X}(x,y)  \arrow[d, "\sim"]\\
\map_{X}(x,y) \arrow[r, "{(\eta_X^{k+1})_*}" swap] 
\arrow[d, "{\eta_{\map_{X}(x,y) }^{k}}" swap] \arrow[ur, "{\eta_{\map_{X}(x,y) }^{k}}" ]
& \map_{T^{k+1} X}(\eta_X^{k+1}(x),\eta_X^{k+1}(y)) \arrow[d, "{\eta_{ \map_{T^{k+1} X}(x,y)}^{k}}"] \\
T^{k} \map_{X}(x,y) \arrow[r, "{T^{k}((\eta_X^{k})_*)}"]
& T^{k} \map_{T^{k+1} X}(\eta_X^{k+1}(x),\eta_X^{k+1}(y)),
\end{tikzcd}
\]
where each component of $\eta^{k}$ is a $\se^{k-1}\cpt(S)$-local equivalence by the induction hypothesis.
Thus the two-out-of-three property tells us that $T^{k+1}\eta_X^{k+1}$ is also a $\se^{k-1}\cpt(S)$-local equivalence on the mapping objects. However, the objects $T^{k+1} X$ and $T^{k+1} T^{k+1} X$ are $\se^{k}\cpt(S)$-fibrant by Corollary \ref{compfibrant}, so $T^{k+1}\eta_X^{k+1}$ is a levelwise weak equivalence by the same argument as the $\se^{k}\cpt(S)$-fibrant induction step. 

Now we have weak equivalences in diagram \ref{wediagram} as indicated:
\[
\begin{tikzcd}[sep=large]
\Map(T^{k+1} X, Z) \arrow[r, "{T^{k+1}}", "\sim" swap] \arrow[d, "(\eta_X^{k+1})^*" swap]
&  \Map(T^{k+1} T^{k+1} X, T^{k+1} Z) \arrow[d, "(T^{k+1}\eta_X^{k+1})^*", "\sim" swap]\\
\Map(X, Z) \arrow[r, "{T^{k+1}}"] \arrow[rd, "(\eta_Z^{k+1})_*" swap, "\sim"]
&  \Map(T^{k+1} X, T^{k+1} Z) \arrow[d, "(\eta_X^{k+1})^*" ]\\
&  \Map( X, T^{k+1} Z),
\end{tikzcd}
\]
which implies that all of the maps in the diagram are weak equivalences by the two-out-of-six property, in particular, the left vertical map. Thus $\eta_X^{k+1}$ is a $\se^{k}\cpt(S)$-local equivalence completing the inductive step.
\end{proof}

The final part of Theorem \ref{suspended} then inductively follows using the relation between DK-equivalences and weak equivalences from Theorem \ref{dkmain}.

\begin{cor}\label{ndk}
Let $X$ be a $\se^n(S)$-fibrant object in $\sps(\Theta^n C)$. Then $\eta_X^n$ is a DK-$n$-equivalence.
\end{cor}

\begin{proof}
By Corollary \ref{underlying}, $\eta_X^n$ is surjective up to homotopy on $k$ cells for $0\leq k\leq n-2$ and thus also up to higher categorical equivalences by Corollary \ref{hneq}, so it suffices to show that it induces DK-1-equivalences on the $(n-1)$-fold mapping objects. Proposition \ref{compwe} tells us that $\eta_X^n$ is a $\se^{n-1}\cpt(S)$-local equivalence, which implies that it induces $\cpt(S)$-local equivalences
\[
\begin{tikzcd}
\map_X^{n-1}(\alpha,\alpha') \arrow[r, "\sim" ] &
\map_{T^n X}^{n-1}(\eta_X^n(\alpha),\eta_X^n(\alpha'))
\end{tikzcd}
\]
between the iterated mapping objects. However, $\cpt(S)$-local equivalences between $\se(S)$-fibrant objects coincide with DK-equivalences by Theorem \ref{dkmain}.
\end{proof}

The following theorem tells us that localizing with respect to the dimension $n$ completeness condition inverts precisely those maps that are DK-$n$-equivalences and essentially identity in dimension $n-2$ and lower.

\begin{thm}\label{dksuspended}
A map between $\se^n(S)$-fibrant objects of $\sps(\Theta^n C)$ is a DK-$n$-equivalence and a levelwise equivalence of the underlying $\Theta_{n-2}$-spaces if and only if it is a $\se^{n-1}\cpt(S)$-local equivalence.
\end{thm}

\begin{proof}
DK-$n$-equivalences satisfy the two-out-of-three property, as do levelwise equivalences; thus an argument identical to Theorem \ref{dkmain} applies, replacing $T^1$ with $T^n$
\end{proof}

Next we consider applying each of the completions successively in the order of decreasing dimension. The fact that the dimension $k$ completion essentially only affects the $(k-1)$-cells substantially simplifies the situation, causing the essential surjectivity in terms of higher categorical equivalences to not appear explicitly. However, when completing in multiple dimensions, this more subtle notion of essential surjectivity becomes essential.

\begin{con}\label{tcompletion}
We define the \emph{total completion} functor $T\colon \sps(\Theta^n C)\to \sps(\Theta^n C)$ as the composite
\[
T :=T^{1}\cdots T^{n-1} T^n .
\]
Note that we also have a natural transformation $\eta\colon \id \Rightarrow T$ with components 
\[
\eta_X:=\eta_{T^{2}\cdots T^n X}^1\cdots \eta_{T^n X}^{n-1} \eta_{X}^n.
\]
\end{con}

The total completion inherits properties similar to the individual completions.

\begin{thm}\label{total}
Let $X$ be a $\se^n(S)$-fibrant object in $\sps(\Theta^n C)$. Then
\begin{enumerate}
\item $T X$ is $\cpt^n(S)$-local,
\item $\eta_X$ is a $\cpt^{n}(S)$-local acyclic cofibration, and
\item $\eta_X$ is a DK-$n$-equivalence.
\end{enumerate}
\end{thm}


\begin{proof}
Let $X$ be $\se^n(S)$-fibrant. By iteratively applying Theorem \ref{suspended}, it follows that $T^{n-k+1}\cdots T^n X$ is $\se^{n-k}\cpt^{k}(S)$-fibrant for all $1\leq k\leq n$. In particular, when $k=n$ we have that $TX$ is $\cpt^{n}(S)$-fibrant. 
Next, note that by Theorem \ref{dksuspended}, for each $1\leq k\leq n$, the map 
\[
\eta_{T^{k+1}\cdots T^n X}^{k}\colon T^{k+1}\cdots T^n X \to T^{k} T^{k+1}\cdots T^n X
\]
is a $\se^{k-1}\cpt^{n-k+1}(S)$-local and thus also a $\cpt^{n}(S)$-local acyclic cofibration.
Thus the composite $\eta_X$ is a $\cpt^{n}(S)$-local acyclic cofibration.

Furthermore, since each $\eta_{T^{k+1}\cdots T^n X}^{k}$ is a DK-$n$-equivalence by Corollary \ref{ndk}, so is  $\eta_X$.
\end{proof}

Finally, we prove Theorem \ref{t3}, which suggests that the conditions in Definition \ref{dkn} correctly describe equivalences of $(\infty,n)$-categories also in the absence of completeness.

\begin{thm}\label{dktotal}
A map between $\se^n(S)$-fibrant objects of $\sps(\Theta^n C)$ is a DK-$n$-equivalence if and only if it is a $\cpt^n(S)$-local equivalence.
\end{thm}

\begin{proof}
The result follows by essentially the same argument as the case $n=1$ in Theorem \ref{dkmain}, replacing $T^1$ with $T$.
\end{proof}

\begin{cor}
A $\se^n(S)$-fibrant object of $\sps(\Theta^n C)$ is $\cpt^n(S)$-fibrant if and only if it is local with respect to the class of DK-$n$-equivalences.
\end{cor}

\newpage
\bibliographystyle{alpha}
\bibliography{ref}

\end{document}